\documentclass[10pt,a4paper]{amsart}

\usepackage{amsmath, amsthm, amssymb}
\usepackage{color, url}

\DeclareMathOperator{\Aut}{Aut}

\DeclareMathOperator{\Isom}{Isom}
\DeclareMathOperator{\Closed}{Closed}
\DeclareMathOperator{\Id}{Id}
\DeclareMathOperator{\graph}{graph}

\newcommand{\N}{\mathbb{N}}
\newcommand{\Z}{\mathbb{Z}}
\newcommand{\R}{\mathbb{R}}

\newcommand{\aeq}{\simeq}
\newcommand{\asm}{\preceq}
\newcommand{\avsm}{\ll}
\newcommand{\diam}{\mbox{diam}}

\newtheorem{theorem}{Theorem}[section]
\newtheorem{alphtheorem}{Theorem}

\newtheorem{lemma}[theorem]{Lemma}
\newtheorem{prop}[theorem]{Proposition}

\newtheorem{corollary}[theorem]{Corollary}

\newtheorem*{theoremBowditch}{Theorem~\ref{theorem:Bowditch}}
\newtheorem*{theoremBowditch+}{Theorem~\ref{Theorem:Bowditch+}}
\newtheorem*{theoremsharply3}{Theorem~\ref{theorem:sharply3intro}}
\newtheorem*{theorem_finite_disconnecting}{Theorem~\ref{theorem:finitedisconnecting}}

\theoremstyle{remark}
\newtheorem{step}{Step}
\newtheorem{question}[theorem]{Question}
\newtheorem{remark}[theorem]{Remark}
\newtheorem{notation}[theorem]{Notation}
\newtheorem{defn}[theorem]{Definition}

\DeclareMathOperator{\PGL}{PGL}
\DeclareMathOperator{\Projline}{P^1}

\newcommand{\actson}{\curvearrowright}

\newcommand{\NN}{\mathbf N}
\newcommand{\ZZ}{\mathbf Z}
\newcommand{\RR}{\mathbf R}
\newcommand{\CC}{\mathbf C}
\newcommand{\FF}{\mathbf F}
\newcommand{\QQ}{\mathbf Q}
\newcommand{\HH}{\mathbf H}

\theoremstyle{remark}
\newtheorem*{claim*}{Claim}
\newtheorem*{question*}{Question}

\newtheorem{example}[theorem]{Example}

\newenvironment{alphtheoremenum}
	{	
		
		\begin{enumerate}}
	{\end{enumerate}}

\title{Locally compact convergence groups and $n$-transitive actions}

\author{Mathieu Carette}
\address{\tt Universit\'e catholique de Louvain, 
IRMP,
Chemin du Cyclotron 2, bte L7.01.01, 
1348 Louvain-la-Neuve,
Belgium}
\email{\tt mathieu.carette@uclouvain.be}

\author{Dennis Dreesen}
\address{\tt School of Mathematics,
University of Southampton;
Highfield, Southampton,
SO17 1BJ, United Kingdom }
\email{\tt dennis.dreesen@soton.ac.uk}

\begin{document}
	

	\thanks{The first author is a Postdoctoral Researcher of the F.R.S.-FNRS (Belgium).}
	\thanks{The major part of this work was conducted while the second author was staying at the universit\'e cathologique de Louvain in Louvain-la-Neuve.  He gratefully acknowledges the support of the F.R.S.-FNRS (Belgium), grant F.4520.11. Currently, the second author is a Marie Curie Intra-European Fellow within the 7th European Community Framework Programme.}
	
\begin{abstract}
All $\sigma$-compact, locally compact groups acting sharply $n$-transitively and continuously on compact spaces $M$ have been classified, except for $n=2,3$ when $M$ is infinite and disconnected. We show that no such actions exist for $n=2$ and that these actions for $n=3$ coincide with the action of a hyperbolic group on a space equivariantly homeomorphic to its hyperbolic boundary. 
We further give a characterization of non-compact groups acting $3$-properly and transitively on infinite compact sets as non-elementary boundary transitive hyperbolic groups. The main tool is a generalization to locally compact groups of Bowditch's topological characterization of hyperbolic groups. Finally, in contrast to the case $n=3$, we show that for $n \geq 4$ if a locally compact group acts continuously, $n$-properly and $n$-cocompactly on a locally connected metrizable compactum $M$, then $M$ has a local cut point. 
\end{abstract}

\maketitle

\setcounter{tocdepth}{1} 
\tableofcontents

\section{Introduction}

In $1998$ \cite{Bowditch}, Bowditch characterized (discrete) hyperbolic groups via their actions on their boundary. It was already known that a non-elementary hyperbolic group $\Gamma$ acts $3$-properly and $3$-cocompactly on its boundary $\partial \Gamma$ which is a compact, perfect, metrizable space. Bowditch showed that conversely, if a discrete group $\Gamma$ acts $3$-properly and $3$-cocompactly on a perfect, compact, metrizable space $M$, then $\Gamma$ is necessarily non-elementary hyperbolic and $M$ is $\Gamma$-equivariantly homeomorphic to $\partial \Gamma$. 

The definition of a $3$-proper, $3$-cocompact action makes sense equally well for locally compact groups as it does for discrete groups (see Section~\ref{sc:Bowditch}). Similarly, the notion of word-hyperbolic group extends beyond the world of discrete groups. Indeed Gromov's work \cite{Gromov} contains ideas which already encompass locally compact groups. Following \cite{Caprace} we define a locally compact group $G$ to be {\bf hyperbolic} if it is compactly generated and word-hyperbolic with respect to the word length metric associated to some compact generating subset. This is in fact equivalent to the group acting continuously, properly, cocompactly and isometrically on a proper hyperbolic geodesic metric space $X$ (see Corollary 2.6 of \cite{Caprace}). The space $X$ is determined up to quasi-isometry and so one can unambiguously define the hyperbolic boundary $\partial G$ of $G$ as the hyperbolic boundary of $X$.

The above leads to the natural question as to whether or not Bowditch's topological characterization generalizes to the class of {\em locally compact} hyperbolic groups. This is not a priori so: locally compact hyperbolic groups exhibit phenomena which do not appear in the discrete context. Most importantly, it is possible for such groups to act transitively on their boundary. On the other hand, discrete hyperbolic groups are countable, and so cannot act transitively on any compact perfect, hence uncountable, space. More warnings, similarities and differences are highlighted following Theorem~\ref{theorem:tree_groups}, as well as in the beginning of Sections~\ref{sc:Bowditch} and~\ref{sc:transitiveconvergence}.

Apart from generalizing Bowditch's characterization to the setting of locally compact groups (see Theorem \ref{theorem:Bowditch} below), we also generalize in Theorem~\ref{Theorem:Bowditch+} a companion characterization of hyperbolicity which makes use of conical limit points (see Section~\ref{sc:Bowditch+} for the definition of conical limit points). Theorem \ref{Theorem:Bowditch+} in the case of discrete groups was proven by Bowditch and independently also by Tukia \cite{Tukia}.

\begin{alphtheorem} \label{theorem:BowBowditch+}
 Let $G$ be a locally compact group acting continuously and $3$-properly on a perfect compactum $M$. 
	\begin{alphtheoremenum} 
		\item Suppose that $G$ acts $3$-cocompactly. Then $G$ is hyperbolic, and there is a $G$-equivariant homeomorphism $\partial G \to M$. \label{theorem:Bowditch}
		\item Suppose that $M$ is metrizable. Then $G$ acts $3$-cocompactly if and only if every $x \in M$ is a conical limit point. \label{Theorem:Bowditch+}
	\end{alphtheoremenum}
\end{alphtheorem}

Similar to the discrete case, a locally compact group $G$ acting continuously and $3$-properly on a compact space will be called a {\bf convergence group}. This terminology refers to an equivalent dynamical formulation of $3$-properness (see Section~\ref{sec:dynamicconvergenceproperty}). If moreover $G$ acts $3$-cocompactly, then we say that $G$ is a {\bf uniform} convergence group. Theorem~\ref{theorem:Bowditch} thus characterizes locally compact hyperbolic groups as uniform convergence groups acting on perfect compact spaces.

Theorem~\ref{theorem:BowBowditch+} will be applied in two different settings. First, we apply it to locally compact groups acting sharply $3$-transitively and continuously on compact spaces as a step towards a classification of all such actions (see Theorem \ref{theorem:sharply3intro}). Next, we consider the class of non-elementary transitive convergence groups acting continuously on infinite compact spaces and we prove that this class of groups coincides with non-elementary boundary transitive hyperbolic groups as studied in \cite{Caprace}  (see Theorem \ref{theorem:transitiveconvergencegroup}).

Before we continue, we introduce some essential notations.
\begin{notation}
Throughout the remainder of this paper, $G$ will be a locally compact group acting continuously on a locally compact space $M$. By convention, we consider the Hausdorff property to be part of the definition of compactness. Often, a compact space is called a {\bf compactum} for short. We let $M^{(n)}$ denote the space of distinct $n$-tuples of $M$. A (not necessarily hyperbolic) locally compact group $G$ is called \textbf{elementary} if $G$ is compactly generated and either $G$ has no end (i.e. $G$ is compact) or if $G$ has $2$ ends (i.e. $G$ acts properly and cocompactly by isometries on the Euclidean line).
\end{notation}
\subsection{Sharply $n$-transitive actions}  
%
%
%
 
A group $G$ acts \textbf{sharply $n$-transitively} on a set $M$ if the induced action on the space of distinct $n$-tuples is free and transitive. In the present paper, we study sharply $n$-transitive actions of $\sigma$-compact locally compact groups on {\em compact} spaces $M$, with a view towards a conjectural classification. Tits \cite{Tits1} showed that if a group $G$ acts sharply $n$-transitively on a set $M$ for $n \geq 4$, then $M$ and thus $G$ are finite (no topology is involved). This solves the cases $n\geq 4$ as sharply $n$-transitive actions on finite sets are classified for all $n$ (see Kerby's book \cite{Kerby}). In Proposition \ref{prop:2-transitiveimpliesfinite}, we use elementary methods to show that there are no sharply $2$-transitive actions of groups on a compact infinite space $M$. So, only the case of sharply $3$-transitive actions on infinite compact spaces remains open. The case where $M$ is connected reduces to the projective groups $PGL_2(\RR)$ and $PGL_2(\CC)$ acting on the real and complex projective line respectively. This follows from the general classification of all faithful 2-transitive actions of locally compact, $\sigma$-compact groups on connected locally compact spaces $M$ due to Tits \cite{Tits2} (this classification was proven again later in \cite{Kramer} using a more modern approach). Connectedness is essential and in fact leads Tits to conclude that $M$ is a manifold and $G$ is a Lie group. It is worth noting however that the class of sharply $3$-transitive actions on {\em disconnected} compact spaces is not empty. Typical examples include the projective group $\PGL_2(k)$ acting on the projective line for $k$ a non-Archimedean local field (e.g. $k = \QQ_p$ for some prime $p$).

It is interesting to note that in all of the aforementioned examples, the projective group $G$ is hyperbolic and that the space $M$ on which it is acting is the hyperbolic boundary $\partial G$. 
Indeed $\PGL_2(\RR)$ is isomorphic to the full isometry group of the hyperbolic plane $\HH^2$ with boundary $S^1 \cong \Projline(\RR)$ and $\PGL_2(\CC)$ is isomorphic to the orientation-preserving isometry group of hyperbolic 3-space $\HH^3$ with boundary $S^2 \cong \Projline(\CC)$. In the non-Archimedean case, the projective line can be identified with the boundary of a locally finite tree, namely the Bruhat-Tits tree associated with $\PGL_2(k)$ (see \cite{Serre}). We will show that any $\sigma$-compact group acting continuously and sharply $3$-transitively on an infinite compact space $M$ is hyperbolic and that the space $M$ is $G$-equivariantly homeomorphic to $\partial G$.

Moreover, conjecturally, the only examples of sharply $3$-transitive actions of $\sigma$-compact groups $G$ on infinite compact spaces are the projective groups described above. Our following result provides strong evidence towards this conjecture.
\begin{alphtheorem}
\label{theorem:sharply3intro}
Let $G$ be a $\sigma$-compact, locally compact group acting sharply $3$-transitively and continuously on an infinite disconnected compact space $M$. Then $G$ acts continuously, properly and vertex-transitively by automorphisms on a locally finite tree $T$ such that there is a $G$-equivariant homeomorphism $f:\partial T \rightarrow M$.
\end{alphtheorem}

The key point on which relies Theorem~\ref{theorem:sharply3intro} is that $G$ acts $3$-properly on $M$. This is shown by noting that, under the hypothesis of the theorem, any orbit map $G \to M^{(3)}$ is a homeomorphism (see Lemma~\ref{lemma:sharplytransisproper}). Since the action of $G$ on itself is proper, so must be its action on $M^{(3)}$. Once we have derived hyperbolicity from Theorem \ref{theorem:Bowditch}, we can then conclude using Theorem $8.1$ in \cite{Caprace}. Alternatively, we could have concluded from Theorem \ref{theorem:tree_groups} below.
\subsection{Uniform convergence groups on the Cantor Set}
Our following result, of independent interest, is well-known for discrete groups.

\begin{alphtheorem} \label{theorem:tree_groups}
	Let $G$ be a non-elementary locally compact group. Then the following are equivalent:
	\begin{alphtheoremenum}
		\item $G$ acts $3$-properly and $3$-cocompactly on a totally disconnected perfect compact space $M$. \label{cond:ucg_cantor}
		\item $G$ is hyperbolic and its boundary $\partial G$ is totally disconnected. \label{cond:hyptdbdry}
		\item $G$ acts continuously, properly and cocompactly on a locally finite tree $T$. \label{cond:treegp}
		\item $G$ is compactly generated and quasi-isometric to a locally finite tree $T$. \label{cond:qitotree}
	\end{alphtheoremenum}
	In that case, the spaces $\partial T, \partial G$ and $M$ are all equivariantly homeomorphic.
\end{alphtheorem}

As a consequence, the space $M$ in the first condition is metrizable, so that the above theorem is in fact a characterization of uniform convergence groups acting on the Cantor Set. Analogously, there is a celebrated characterization of uniform convergence groups on the circle as those locally compact groups acting properly and cocompactly by isometries on the hyperbolic plane \cite{Tuk88,Gabai,CassJung,Hin}, see also \cite[Theorem 5.D.1]{Cornulier}.

We stress here that the family $\mathcal T$ of locally compact groups satisfying the equivalent conditions of Theorem~\ref{theorem:tree_groups} is much richer than its subfamily consisting of discrete groups. Indeed, if $G$ is discrete, then the conditions of Theorem \ref{theorem:tree_groups} are equivalent to $G$ being finitely generated virtually \{free non-abelian\} \cite[II.2.6]{Serre}. In particular, such groups are always linear, and have many normal subgroups. On the other hand the full automorphism group of a $d$-regular tree is not linear, and has a simple subgroup of index 2 \cite{Titstree}. We refer the reader to \cite{CapDeMedts} for more on linearity. There are also groups $G$ in $\mathcal T$ which do not admit any free lattice. Indeed, there exist such groups which are non-unimodular \cite[(4.12)]{BassKulk}, so that they do not admit any lattice. In fact, this is the only obstruction: if $G$ is assumed to be unimodular, then it admits a uniform lattice \cite{BassKulk}, see also \cite{BassLub}.

The last three conditions in Theorem~\ref{theorem:tree_groups} were already known to be equivalent. We refer to the recent paper of Cornulier \cite{Cornulier} for a thorough discussion of these equivalences and their history.

\subsection{Transitive convergence groups} In \cite{Caprace}, the authors study boundary transitive hyperbolic groups.
Using our Theorem \ref{theorem:BowBowditch+}, we are able to give another characterization of these groups.
\begin{alphtheorem} \label{theorem:transitiveconvergencegroup}
Let $G$ be a $\sigma$-compact, locally compact, non compact group acting continuously and transitively as a convergence group on an infinite compact space $M$. Then $G$ is a boundary-transitive hyperbolic group and $M$ is $G$-equivariantly homeomorphic to $\partial G$.  Consequently, $G$ has a maximal compact normal subgroup $K$ such that one of the following holds:
\begin{itemize}
	\item $G/K$ is the isometry group or the orientation preserving isometry group of a rank one symmetric space $X$.
	\item $G/K$ acts continuously, properly, cocompactly and faithfully on a locally finite tree $X$ such that the induced action on the boundary $\partial X$ is 2-transitive.
\end{itemize}
Moreover, $M$ is equivariantly homeomorphic to the visual boundary of $X$.
\end{alphtheorem}
The group $K$ above is the kernel of the action of $G$ on its boundary, so that $K = \{1\}$ if and only if $G$ acts faithfully on $M$. Following \cite{Caprace}, we call the groups $G/K$ as above \textbf{standard rank one groups}.
\subsection{$n$-proper actions} From Section \ref{subsc:n=2}, it follows that a $\sigma$-compact group acting $2$-properly, $2$-cocompactly on a metrizable compactum must itself be compact. Having also studied $3$-proper, $3$-cocompact actions, we next look at the case $n>4$. In view of Tits' result that there does not exist any sharply $n$-transitive infinite group for $n \geq 4$, the following (which reformulates a question initially asked by G. Mess to B. Bowditch) is of particular interest.
\begin{question} Does there exist an infinite locally compact group $G$ acting continuously, $n$-properly and $n$-cocompactly on a compact, perfect metrizable space $M$ for $n \geq 4$?
\end{question} 
The following theorem is an obstruction to such an action, showing that the space $M$ must satisfy a certain ``disconnectedness property''. Precisely, we obtain:   
\begin{alphtheorem} \label{theorem:finitedisconnecting}
	Let $n \geq 4$. Suppose that $G$ is a locally compact group acting continuously, $n$-properly and $n$-cocompactly on a compact, locally connected, perfect and metrizable space $M$. Then there is a set $P \subset M$ of cardinality at most $\lfloor \frac{n-1}{2} \rfloor$ such that $M \backslash P$ is not connected. 
\end{alphtheorem}

Gerasimov in \cite{Gerasimov} refers to this result already appearing as the main result in an unpublished preprint of Yaman \cite{Yaman}. Our proof uses Gerasimov's reformulation of $n$-properness, but is otherwise elementary.	

Theorem \ref{theorem:finitedisconnecting} emphasizes the different behaviours for $n=4$ and $n=3$ respectively. Indeed, for $n=4$, the local connectedness of $M$ implies that $M$ has a global cut point. However, the solution of the cut-point conjecture (see \cite{B2}, \cite{SW1}) shows that whenever the boundary of a discrete hyperbolic group is connected, then it has no global cut points and it is locally connected (\cite{BM}).


\subsection{Further questions and remarks} We remark that the methods of this paper also provide an alternative proof of a theorem of Nevo \cite{Nevo} about locally finite graphs with infinitely many ends and with a non-compact automorphism group acting transitively on the space of ends (see Section \ref{sec:end-transitivegraphs}).\\

Compactness of the space $M$ in Theorem~\ref{theorem:sharply3intro} cannot be dropped. Indeed, the (countable) group $G = \PGL_2(\QQ)$ with the discrete topology acts sharply 3-transitively on the (discrete) projective line $M = \QQ \cup \{\infty\}$, which is not compact, and hence not homeomorphic to the boundary of a locally finite tree.

The fact that $G$ in Theorem~\ref{theorem:sharply3intro} is hyperbolic, and in particular compactly generated, depends crucially on the fact that $M$ is compact. We are thus led to ask the following question, which is true for connected $M$ (as follows from Tits' classification \cite{Tits2}).
\begin{question} Suppose $G$ is a non-discrete, compactly generated locally compact group, acting continuously and sharply $3$-transitively on a locally compact space $M$. Is $M$ necessarily compact?
\end{question}

The next remark shows that the hypothesis of $\sigma$-compactness in Theorem \ref{theorem:sharply3intro} cannot be omitted either.
\begin{remark} \label{rmk:discretesharply3trans} Theorem~\ref{theorem:sharply3intro} is false if $G$ is not assumed to be $\sigma$-compact. Indeed, the group $\PGL_2(\QQ_p)$ endowed with the discrete topology acting on the projective line (with the usual topology) satisfies all other hypotheses. However that group is not $\sigma$-compact, and so it cannot act properly on any $\sigma$-compact space, and in particular on a locally finite tree.
\end{remark}
The $\sigma$-compactness is used to establish that any orbit map $G\rightarrow M^{(3)}$ is a homeomorphism.
In view of Remark \ref{rmk:discretesharply3trans}, it is natural to ask whether the only non-$\sigma$-compact sharply $3$-transitive groups of homeomorphism of an infinite compact set come from an inappropriate choice of topology.

\begin{question} Suppose an abstract group $G$ acts sharply $3$-transitively by homeomorphisms on a compact set $M$. Can $G$ be endowed with a topology which turns it into a topological group acting continuously on $M^{(3)}$ such that any orbit map $G \to M^{(3)}$ is a homeomorphism?
\end{question}

The $\sigma$-compactness in Theorem~\ref{theorem:transitiveconvergencegroup} is only used to show that $M$ is metrizable, so that we can use Theorem~\ref{Theorem:Bowditch+}. Removing the hypothesis thus amounts to asking whether metrizability of $M$ is needed in Theorem~\ref{Theorem:Bowditch+}, the proof of which relies heavily on a metric on $M$.

\subsection{Structure of the paper} The paper is organized as follows. In Section \ref{sec:tools} we recall the ingredients of the proof of Bowditch's topological characterization of hyperbolicity in the discrete case \cite{Bowditch}. In Sections \ref{sc:Bowditch} and \ref{sc:Bowditch+}, we use these ingredients to show that uniform convergence groups are hyperbolic and to characterize $3$-cocompactness in terms of conical limit points, thus proving both parts of Theorem~\ref{theorem:BowBowditch+}. In Section~\ref{sc:sharply-n-transitive}, we give an overview of known results about sharply $n$-transitive groups and we prove Theorem~\ref{theorem:sharply3intro}. We prove the characterization of transitive convergence groups (Theorem~\ref{theorem:transitiveconvergencegroup}) in Section~\ref{sc:transitiveconvergence}. Uniform convergence groups on the Cantor Set are characterized by proving Theorem~\ref{theorem:tree_groups} in Section~\ref{sec:tree_groups}. Finally in Section~\ref{sec:nproper}, we concentrate on $n$-proper actions for $n \geq 4$ and prove Theorem~\ref{theorem:finitedisconnecting}.
	
	\section*{Acknowledgements}
	The authors are grateful to Pierre-Emmanuel Caprace for suggesting the problem, as well as for helpful discussions. The authors also thank him and Ralf K\"ohl for motivating discussions which lead to Theorem \ref{theorem:transitiveconvergencegroup}.

\section{The main tools} \label{sec:tools}

The proof of Theorem~\ref{theorem:BowBowditch+} relies on the same machinery as in \cite{Bowditch}. We start by introducing the main ingredients and explain the interplay between them. We refer the reader to \cite{Bowditch} for more details.

The following notation is used throughout the paper.\\
{\bf Notation 1.} For any real numbers $k,p,q\in \R$, we write $p\aeq_k q, p\asm_k q$ and $p\avsm_k q$ to mean, respectively, $\lvert p-q \rvert \leq k, p\leq q+k$ and $p\leq q-k$.
\subsection{Crossratios}
\subsubsection{Basic definitions and examples}
A metric tree $\tau$ is a tree where each edge is dedicated a length $r\in (0,\infty)$. The length of a path in this tree is defined as the sum of the lengths of the edges contained in this path. The distance between vertices  $x,y\in \tau$, denoted $d_\tau(x,y)$, is defined as the lenght of the unique path $[x,y]$ without backtracking that connects $x$ to $y$.

Let us define the following $4$-ary operation:
\begin{equation}
 \forall x,y,z,w\in \tau, \mbox{ we set} (xy\mid zw)_\tau := \frac{1}{2} \max\{0,d(x,z)+d(y,w)-d(x,y)-d(z,w)\} .
 \label{eq:treecrossratio}
 \end{equation}
It is easy to check that the above formula represents the distance between the paths $[x,y]$ and $[z,w]$ in $\tau$. This definition benefits from quite some symmetry: for every $x,y,z,w\in \tau$, we have $(xy\mid zw)_\tau=(yx\mid zw)_\tau=(zw\mid xy)_\tau$. 

\begin{defn}
Let $M$ be a set. A map $M^{(4)}\rightarrow [0,+\infty), (x,y,z,w)\mapsto (xy\mid zw)$, satisfying the above mentioned symmetry conditions (i.e. $(xy\mid zw)=(yx\mid zw)=(zw\mid xy)$, $\forall (x,y,z,w)\in M^{(4)}$), is called a {\bf crossratio}.
\end{defn}

\begin{defn}[see \cite{Bowditch}]
Given $k\in \R^+$. We say that a crossratio $(..\mid ..)$ on a set $M$ is {\bf k-hyperbolic} if it satisfies the following axioms:
\begin{enumerate}
\item If $F\subseteq M$ is a $4$-element subset, then we can write $F=\{x,y,z,w\}$ with $(xz\mid yw)\aeq_k 0$ and $(xw\mid yz)\aeq_k 0$.
\item If $F\subseteq M$ is a $5$-element subset, then we can write $F=\{x,y,z,w,u\}$ with
\[ \begin{array}{c}
(xy\mid zu) \aeq_k (xy\mid wu) \\
(xu\mid zw) \aeq_k (yu\mid zw) \\
(xy\mid zw) \aeq_k (xy\mid zu)+(xu\mid zw),
\end{array} \]
and with $(ab\mid cd) \aeq_k 0$ in all other case where $a,b,c,d\in F$ are distinct (allowing for the symmetries of the crossratio). 
\end{enumerate}
A crossratio is called {\bf hyperbolic} if it is $k$-hyperbolic for some $k\in \R^+$.
\label{defn:hypcro}
\end{defn}
Note that in the case of a metric tree $\tau$, the crossratio $(..\mid..)_\tau$ as defined in Equation \ref{eq:treecrossratio}, is $0$-hyperbolic. 
In fact, as the above definition indicates, given a hyperbolic crossratio, there is a constant such that when you restrict the crossratio to any $4$- or $5$-point subset, then it coincides up to this constant with the crossratio $(..\mid ..)_\tau$ of a metric tree. It turns out that the following more general result is valid.
\begin{theorem}[Theorem 2.1 in \cite{Bowditch}]
For all $n\in \N$, there is some constant $h(n)$, depending only on $n$, such that if $(..\mid ..)$ is a $k$-hyperbolic crossratio defined on a set $M$ of cardinality $n$, then we can embed $M$ in a metric tree $\tau$ such that for all $(x,y,z,w)\in M^{(4)}$, we have $\lvert (xy\mid zw) -(xy\mid zw)_\tau \rvert \leq kh(n)$. \label{theorem:approxtree}
\end{theorem}
We refer the reader to \cite{Bowditch} for a proof and to Example \ref{example:crossratiocantor} below for an example.
\subsubsection{The crossratio topology}
Assume that $(..\mid ..)$ is a hyperbolic crossratio on a set $M$ and fix two distinct elements $a,b\in M$. For any $r\geq 0$ and any $x\in M\backslash \{a,b\}$, we define
\[ D_{ab}(x,r)=\{x\} \cup \{y\in M\backslash \{a,b,x\} \mid (ab\mid xy) \geq r \} .\] It turns out that the sets $\{D_{ab}(x,r)\mid r\geq 0\}$ form a base of neighbourhoods of $x$, relative to a certain metrizable topology on $M\backslash \{a,b\}$. To see where this metric comes from, remark that $D_{ab}(x,r)\subset D_{ab}(x,\tilde{r})$ for $\tilde{r}<r$. With the idea of a visual metric in the back of our minds, we could intuitively try to define a distance between distinct points $x,y\in M\backslash \{a,b\}$ as $\lambda^{-(ab\mid xy)}$ where $\lambda$ is a certain positive real number. It turns out that this expression does not in general define a metric. When $\lambda>1$ is sufficiently close to $1$, it {\em does} however define a quasi-ultrametric. One can obtain a metric by modifying the quasi-ultrametric as in \cite{Gromovhyperbolicgroups}.

One can show that for $a,b,c\in M$ distinct, the above defined topologies on $M\backslash \{a,b\}$ and $M\backslash \{a,c\}$ coincide. Indeed, it follows from Definition \ref{defn:hypcro} that, up to a fixed constant, $(ac\mid xy)\leq (ab\mid xy)+p$ and $(ab\mid xy)\leq (ac\mid xy)+q$ where $p=(ac\mid bx)$ and $q=(ab\mid cx)$. Consequently, $D_{ac}(x,r) \subseteq D_{ab}(x,r-p)$ and $D_{ab}(x,r)\subseteq D_{ac}(x,r-q)$.

Applying this observation twice, we get that for $a\neq b \in M$ and $c\neq d \in M$, the topologies on $M\backslash \{a,b\}$ and $M\backslash \{c,d\}$ coincide. Consequently, we obtain a well defined metrizable topology on $M$, which we call the {\bf crossratio topology}.
\begin{defn}
A hyperbolic crossratio $(..\mid..)$ on a topological space $(M,\mathcal{S})$ is {\em compatible with} $\mathcal{S}$ if the crossratio topology and $\mathcal{S}$ coincide.
\end{defn}
\begin{remark}
If $(M,\mathcal{S})$ is equipped with a compatible hyperbolic crossratio, then $\mathcal{S}$ is equivalent to the crossratio topology and so in particular, it is a metrizable topology.
\end{remark}
\subsubsection{Other properties of crossratios}
For the remainder of this text, we will be interested in {\em perfect, hyperbolic path crossratios}. We introduce these concepts here.

\begin{defn}
A hyperbolic crossratio is called {\bf perfect} if for any distinct $x,y,z\in M$, we can find a sequence $(x_i)_{i \in \N}$ in $M\backslash \{x,y,z\}$ such that $(yz\mid xx_i)\to \infty$. Said differently, a crossratio is perfect if the crossratio topology is perfect.
\end{defn}
\begin{defn}
A hyperbolic crossratio is a {\bf path crossratio} if there exists $p\in \R^+$ such that given any distinct $x,y,z,w\in M$ there is a finite sequence of points, $y=u_0,u_1,\ldots ,u_n=w$ of $M$ such that $(xu_i\mid zu_j)\aeq_p j-i$ for all $i,j\in \{1,2,\ldots, n\}$ with $i<j$.
\end{defn}
By enlarging the hyperbolicity constant or by enlarging $p$, we can assume that $p$ is equal to the constant of hyperbolicity.

The following lemma (see also Lemma $2.8$ \cite{Bowditch}) will be needed in Section \ref{sc:Bowditch+}. We add more details to the proof given by Bowditch.
\begin{lemma}
Let $(..\mid..)$ be a perfect $k$-hyperbolic path crossratio on a set $M$. There exists a constant $p$ such that for all distinct $a,b,c\in M$,
there is a bi-infinite sequence, $(x_i)_{i\in \Z}$ in $M\backslash \{a,b\}$ such that $\forall i<j\in \Z:\ (bx_i\mid ax_j)\aeq_p j-i$ and $x_0=c$.
\label{lm:infinitegeodesic}
\end{lemma}
\begin{proof}
Since $M$ is perfect (and the crossratio topology is first countable), we can find a bi-infinite sequence $(y_i)_i$ in $M\backslash \{a,b\}$ with $y_i\to a$, $y_{-i}\to b$ and $y_0=c$. Passing to a subsequence but retaining $y_0$, we can assume that for all $0\leq j<i$, we have $(by_0\mid ay_j)\avsm_k (by_0\mid ay_i)$ and $(ay_0\mid by_{-j})\avsm_k (ay_0\mid by_{-i})$. Using Theorem \ref{theorem:approxtree}, it is clear that $0\avsm_k (by_j\mid ay_i)$ for $j<i$. 

For each $i \in \ZZ$, we define $k(i) = 0$ if $i=0$; $k(i) =-\left\lceil (a y_0 | y_i b) \right\rceil$ if $i < 0$ and $k(i) =   \left\lceil (a y_i | y_0 b) \right\rceil$ if $i > 0$ (here, $\left\lceil z \right\rceil$ denotes the smallest integer $\geq z$). Thus we can assume that $k$ is an increasing injection $\ZZ \to \ZZ$ such that for $j < i$ we have that $(a y_i | y_j b) \simeq k(i) - k(j)$. We now define a new sequence $(x_j)_{j\in \ZZ}$ by interpolating between the $y_i$ using the path property as follows : for each $i \in \ZZ$ define $y_i = x_{k(i)}, x_{k(i)+1}, x_{k(i)+2}, \ldots, x_{k(i+1)} = y_{i+1}$ so that $(a x_m | x_n b) \simeq m-n$ for $k(i) \leq n < m \leq k(i+1)$. We claim that the sequence is such that $(a x_m | x_n b) \simeq m - n$ for any integers $m > n$. If there is some $i$ such that $k(i) \leq n < m \leq k(i+1)$, then this follows from the definition of the $x_j$. Otherwise there are $j \leq i$ such that $k(j-1) < n < k(j)$ and that $k(i) < m < k(i+1)$. Then using hyperbolicity of the cross-ratio (in the form of the triangle inequality $(cd | e f) \asm (cd | y f) + (c y | e f)$), we have 
		\[ (a x_m | x_n b) \asm (a x_m | x_{k(i)}b) + (a x_{k(i)} | x_{k(j)} b) + (a x_{k(j)} | x_n b) \simeq m-n .\]
		On the other hand we have
		\[ (a x_m | x_n b) \succeq (a x_{k(i+1)} | x_{k(j-1)} b) - (a x_{k(i+1)} | x_m b) - (a x_n | x_{k(j-1)} b) \simeq m-n.\]
		Thus we indeed have $(a x_m | x_n b) \simeq m-n$ as desired.
\end{proof}
\subsection{Quasimetrics \label{subsc:Quasimetrics}}
In this paragraph, we introduce quasimetrics and elaborate on the interplay between  quasimetrics and their associated crossratio.
\begin{defn}
A k-quasimetric $\rho$, on a set $Q$ is a function $\rho:Q^2\rightarrow [0,\infty)$ satisfying $\rho(x,x)=0, \rho(x,y)=\rho(y,x)$ and $\rho(x,y)\leq \rho(x,z)+\rho(z,y)+k$ for all $x,y,z\in Q$.
A {\bf quasimetric} is a $k$-quasimetric for some $k$. We refer to $k$ as the quasimetric constant.
\end{defn}
\begin{remark}
 Note that for $k=0$, $\rho$ is simply a pseudometric.
\end{remark}
As we relax the triangle inequality by an additive constant, we lose the possibility of associating a topology to a quasimetric: the collection of balls in a quasimetric space does no longer form a base for a topology.

On the other hand, 
we can still define {\em hyperbolicity} and the {\em hyperbolic boundary}. To this end, recall that in Equation \ref{eq:treecrossratio}, we used the metric on a tree to define a crossratio on that tree. We can do the same thing for any quasimetric space $(Q,\rho)$: given $x,y,z,w\in Q$, we set
\[\begin{array}{ccl}
(xy\mid zw)_\rho &= &\frac{1}{2}[\max\{\rho(x,y)+\rho(z,w),\rho(x,w)+\rho(y,z), \rho(y,w)+\rho(x,z)\}\\
& & - (\rho(x,y)+\rho(z,w))].
\end{array}\]
Note that $(xy\mid zw)_\rho=(yx\mid zw)_\rho=(zw\mid xy)_\rho$ and so 	$(..\mid..)_\rho$ is a well defined crossratio. 
\begin{defn}
We say that the quasimetric space $(Q,\rho)$ is {\bf hyperbolic} if the associated crossratio $(..\mid..)_\rho$ is a hyperbolic crossratio. We refer to the constant involved as the {\em hyperbolicity constant}.
\end{defn}
The classical ``four-point'' characterization of hyperbolicity for metric spaces still makes sense for quasimetrics. It can be shown that this definition is equivalent to the definition given above. 

Closely following Bowditch's exposition, we will now introduce the {\bf hyperbolic boundary}.
\begin{defn}
A {\bf $k$-geodesic segment} (connecting $x_0$ to $x_n$) is a finite sequence of points $x_0,x_1,\ldots ,x_n$ with $\rho(x_i,x_j)\aeq_k \lvert i-j\rvert$ for all $i,j\in \{1,2,\ldots ,n\}$. Similarly, one defines the notions of {\em k-geodesic ray} (indexed over $\N$) and {\em bi-infinite k-geodesic} (indexed over $\Z$).

We say that a quasimetric is a {\bf path quasimetric} if there is some $k\geq 0$ such that every pair of points can be connected by a $k$-geodesic segment. By enlarging $k$ if necessary, we can assume that $k$ equals the quasimetric constant. We simply write {\bf geodesic} instead of {\em $k$-geodesic} unless there is a chance for ambiguity.
\end{defn}
Let $(Q,\rho)$ be a hyperbolic path quasimetric space and fix a basepoint $a\in Q$. Given two geodesic rays emanating from $a$, we say that they are {\em parallel} if they remain at bounded distance from each other. We define $\partial Q$ as the set of parallel classes of geodesic rays that emanate from $a$.

Given $x\in \partial Q$, let $(x_i)_i$ be a geodesic ray in the class of $x$. Fix any $r>0$. For any given natural number $n$, we define $D(n)$ as the set of all $y\in Q\cup \partial Q$ such that some geodesic connecting $a$ to $y$ (i.e. a ray in the class of $y$ if $y\in \partial Q$) meets $N_\rho(x_n,r):=\{z\in Q\mid \rho(x_n,z)<r\}$. The collection $\{D(n)\mid n\in \N\}$ defines a base of neighbourhoods of $x\in Q\cup \partial Q$. Setting the topology on $Q$ to be discrete, one obtains a well defined topology on $Q\cup \partial Q$. As it turns out, this topology is actually induced by a metric and is thus Hausdorff.

There is a natural way to define a crossratio on $\partial Q$. Given $(x_1,x_2,x_3,x_4)\in (\partial Q)^{(4)}$, we can find pairwise disjoint open sets $O_1,O_2,O_3,O_4\subset Q\cup \partial Q$ with $x_i\in O_i$ and such that if $y_i,z_i\in Q\cap O_i$, then $(y_1,y_2,y_3,y_4)_\rho \aeq (z_1,z_2,z_3,z_4)_\rho$. Let us define $(O_1,O_2,O_3,O_4)$ as the supremum of $(y_1,y_2,y_3,y_4)_\rho$ over the $y_i\in O_i\cap Q$. Finally, define $(x_1,x_2,x_3,x_4)_\rho$ to be the limit of $(O_1,O_2,O_3,O_4)$ as each set $O_i$ shrinks to the point $x_i$. It is easy to check that this equips $\partial Q$ with a hyperbolic crossratio and that the topology on $\partial Q$ coincides with the crossratio topology.

In fact, by holding fixed any of the points in the above construction, one can also define $(ab\mid cx)_\rho, (ab\mid xy)_\rho$ and $(ax\mid yz)_\rho$ for $a,b,c\in Q, x,y,z\in \partial Q$. This equips $Q\cup \partial Q$ with a hyperbolic crossratio.

In this paragraph, we have seen that a hyperbolic quasimetric $\rho$ on a set $Q$ gives rise to a hyperbolic crossratio on $Q$.
The next paragraph describes some kind of a converse, namely that a hyperbolic crossratio on a set $M$ induces a hyperbolic quasimetric on {\bf the set of distinct triples}
\[ M^{(3)}=\{(x,y,z)\in M \mid x\neq y, x\neq z, y\neq z \} \]
 {\bf of $M$}.
\subsection{Crossratios induce quasimetrics}
Let us start with a hyperbolic crossratio $(..\mid..)$ on a set $M$ such that $(xy\mid xy)=(xy\mid xz)=0$ for all distinct $x,y,z\in M$. Given any two distinct triples $X=(x_1,x_2,x_3), Y=(y_1,y_2,y_3)$ in $M^{(3)}$, we define $\rho(X,Y)\in [0,\infty)$ as
\begin{equation}
\rho(X,Y)=\max\{(x_ix_j\mid y_my_n)\mid i,j,m,n \in \{1,2,3\}, i\neq j, m\neq n\}.
\label{equation:quasimetric}
\end{equation}

There is another, more intuitive way of defining $\rho$. For any two triples $X=(x_1,x_2,x_3), Y=(y_1,y_2,y_3)\in M^{(3)}$, take an approximating tree $(\tau,d_\tau)$ for the finite set
$\{x_1,x_2,x_3,y_1,y_2,y_3\}$ as in Theorem \ref{theorem:approxtree}. Define $x$ and $y$ to be the medians in $\tau$ of the triples $(x_1,x_2,x_3)$ and $(y_1,y_2,y_3)$. One can show the following lemma.
\begin{lemma}[Lemma 4.1 in Bowditch]
There exists a constant $k\in \R^+$ such that for all triples $X,Y\in M^{(3)}: \rho(X,Y)\aeq_k d_\tau(x,y)$.
\end{lemma}
\begin{example}
Let $M$ be the Cantor set, viewed as the hyperbolic boundary of the free group $F_2=<e_1,e_2>$. Given $(x,y,z,w)\in M^{(4)}$, we define $(xy\mid zw)$ as the distance between the bi-infinite geodesics $[x,y]$ and $[z,w]$ in the Cayley graph Cay$(F_2,\{e_1,e_2\})$ (here, every edge is assigned length $1$). One checks that this definition gives a hyperbolic crossratio on $M$.

In order to check Theorem \ref{theorem:approxtree} in this case, let $F$ be any finite set of points of $M$. For every two distinct elements $x,y\in F$, we denote the bi-infinite geodesic in Cay$(F_2,\{e_1,e_2\})$ connecting these two points by $\gamma_{xy}$. Given two such bi-infinite geodesics $\gamma_{xy}, \gamma_{zw}$ with $(x,y,z,w)\in M^{(4)}$, there are two uniquely defined points $c_{xy,zw}\in \gamma_{xy}$ and $c_{zw,xy}\in \gamma_{zw}$  minimizing the distance between $\gamma_{xy}$ and $\gamma_{zw}$ and of minimal length (i.e. we choose the points as close to $1\in F_2$ as possible). Let $T$ be a finite subtree of $F_2$ containing $1\in F_2$ and all of these points $c_{ab,cd}$ with $a,b,c,d\in F$ all distinct. For every $x\in F$, let $t_x$ be the point in $T$ which lies furthest on the geodesic ray connecting $1$ to $x$ in the Cayley graph of $F_2$. Identifying the points $x\in F$ with their associated $t_x\in T$, one checks that $T$ satisfies the conditions of Theorem \ref{theorem:approxtree}.

Take points $X=(x_1,x_2,x_3)$ and $Y=(y_1,y_2,y_3)$ in $M^{(3)}$. An approximating tree for points $x_1,x_2,x_3,y_1,y_2,y_3\in M$ is in particular a subtree of the Cayley graph of $F_2$. This tree contains the intersection point $x$ of the bi-infinite geodesics $[x_1x_2], [x_2x_3]$ and $[x_1x_3]$. Similarly, it contains the intersection point $y$ of $[y_1y_2], [y_2,y_3]$ and $[y_1y_3]$. In this case $\rho(X,Y)=d_{F_2}(x,y)$.
\label{example:crossratiocantor}
\end{example}
We can now formulate the following striking fact.
\begin{theorem}[Proposition $4.7$ in \cite{Bowditch}]
Suppose that $M$ is a perfect compactum with compatible hyperbolic path crossratio $(..\mid..)$. Then $\rho$ as defined in Equation (\ref{equation:quasimetric}) is a hyperbolic path quasimetric on the space of distinct triples $Q$ of $M$ in such a way that $M$ can be naturally identified by a homeomorphism with the hyperbolic boundary $\partial Q$. Moreover, the crossratios $(..\mid..)$ and $(..\mid..)_\rho$ on $\partial Q$ differ by at most an additive constant.
\label{theorem:crossratiotoquasimetric}
\end{theorem}
Bowditch's embedding $f:M\rightarrow \partial Q$ is defined as follows. Given $a\in M$, choose any $b\in M\backslash \{a\}$ and consider a sequence as in Lemma \ref{lm:infinitegeodesic}, i.e. a sequence $(x_i)_i$ such that $(bx_i\mid ax_j)\aeq j-i$ for all integers $i<j$. Set $X_i=(b,a,x_i)$. By the definition of $\rho$, one sees easily that $\rho(X_i,X_j)\aeq \lvert i-j \rvert$, i.e. that $(X_i)_i$ is a $k$-geodesic ray in $(Q,\rho)$. Bowditch defines $f(a)\in \partial Q$ as the class of the ray $(b,a,x_i)_i$. In Lemma $4.5$ of \cite{Bowditch}, it is shown that the choice of $b$ in this definition is irrelevant: more precisely, if $c\in M\backslash \{a,b\}$, then the rays $(b,a,x_i)_i$ and $(c,a,x_i)_i$ are parallel.

For later use, we mention a particular property of this embedding. Given $3$ points $f(x),f(y),f(z)$ in the hyperbolic boundary $\partial Q\cong M$ of $Q$, we say that $c\in Q$ is a center of $f(x),f(y),f(z)$ if $(f(x)f(y)\mid cc)_\rho \simeq (f(x)f(z)\mid cc)_\rho \simeq (f(y)f(z)\mid cc)_\rho \simeq 0$.
\begin{lemma}[cfr Lemma $4.6$ in \cite{Bowditch}]
Given distinct $x,y,z\in M$, the triple $(x,y,z)$ is a centre, in $Q$, of the triple of ideal points $f(x),f(y),f(z)$.
\label{lemma:centre}
\end{lemma}

Assume that $M$ is a perfect compactum. Referring forward to Lemma \ref{lemma:SchwartzMilnor}, the main idea for proving Theorem \ref{theorem:Bowditch} becomes clear. It shows the importance of being able to construct a $G$-invariant hyperbolic path crossratio on $M$, compatible with the topology on $M$. The key ingredient to construct such a crossratio is called an ``annulus''.
\subsection{Annulus systems}
\begin{defn}
An {\bf annulus} $A$ is an ordered pair $(A^-,A^+)$ of disjoint closed subsets of $M$ such that $M\backslash (A^-\cup A^+)$ is non-empty. An {\em annulus system} $\mathcal{A}$ is a set of such annuli.
\end{defn}

Let us introduce the same notations as in \cite{Bowditch}. First, given an annulus $A=(A^-,A^+)$, we write $-A=(A^+,A^-)$. A system of annuli $\mathcal{A}$ is called {\bf symmetric} if $A\in \mathcal{A} \Leftrightarrow -A\in \mathcal{A}$.

For a closed set $K\subset M$, we write $K<A$ whenever $K\subseteq \mbox{int}(A^-)$ (here, int$(A^-)$ stands for the interior of $A^-$). We write $A<K$ when $K\subseteq \mbox{int}(A^+)$. If $A$ and $B$ are two annuli, then we write $A<B$ if $M\backslash \mbox{int}(A^+) <B$.

Now, fix an annulus system $\mathcal{A}$ and let $K,L$ be closed subsets of $M$. If $A_1,A_2,\ldots ,A_n$ satisfy
\[ K<A_1<A_2<\ldots <A_n<L, \]
then we refer to $(A_i)_i$ as a sequence of $n$ {\bf nested annuli seperating} $K$ and $L$.
We define $(K\mid L)$ as the supremum in $\N \cup \{\infty\}$ of all numbers $n\in \N$ such that there is a sequence of $n$ nested annuli seperating $K$ and $L$.

The most interesting case occurs when $K$ and $L$ both consist of two elements. To be precise, let $x,y,z$ and $w$ be $4$ distinct elements of $M$ and denote $(xy\mid zw):=(\{x,y\}\mid \{z,w\})$. It follows immediately from the definitions that the map $[(x,y,z,w)\mapsto (\{x,y\}\mid \{z,w\})]:M^{(4)}\rightarrow \N\cup \{\infty\}$ defines a crossratio $(..\mid ..)$ on $M$.
\begin{theorem}[see \cite{Bowditch}, Section 6]
Let $M$ be a perfect compactum and let $\mathcal{A}$ be a symmetric annulus system on $M$ which satisfies the following conditions:
\begin{enumerate}
\item (A1): If $x\neq y$ and $z\neq w$, then $(xy\mid zw)<\infty$.
\item (A2): There is some $k\geq 0$ such that there are no four points $x,y,z,w\in M$ with $(xz\mid yw)>k$ and $(xw\mid yz)>k$.
\end{enumerate}
Then, the map $[(x,y,z,w)\mapsto (xy\mid zw):M^{(4)}\rightarrow [0,\infty)],$ defined as the maximum number of nested annuli of $\mathcal{A}$ seperating $\{x,y\}$ and $\{z,w\}$, is a hyperbolic path crossratio on $M$.

Moreover, if $\mathcal{A}$ also satisfies
\[ (A3): \mbox{If }x,y,z\in M \mbox{ are distinct, then } (x\mid yz)=\infty, \]
then the crossratio is also compatible with the topology on $M$.
\label{theorem:annulitocrossratios}
\end{theorem}
Referring to Theorem \ref{theorem:crossratiotoquasimetric}, the above result shows the importance of finding a symmetric annulus system $\mathcal{A}$ on $M$ which satisfies conditions $(A1), (A2), (A3)$ above. We will do this twice: once in Section \ref{sc:Bowditch} using the $3$-cocompactness of the $G$-action on $M$ and once in Section \ref{sc:Bowditch+} using the metrizability of $M$.
	
\section{Locally compact uniform convergence groups \label{sc:Bowditch}}

Let $M$ be a compact space, and let $G$ be a locally compact group acting \textbf{continuously} on $M$, i.e. such that the map $G \times M \to M, (\gamma,x)\mapsto \gamma\cdot x$ is continuous.
 For $n \in \NN$, let $M^{(n)}$ be the space of (pairwise) distinct $n$-tuples 
 \[ M^{(n)}=\{(x_1,x_2,\ldots ,x_n)\mid x_i\neq x_j \ \forall i\neq j \}. \] We endow $M^n$ with the product topology, and $M^{(n)}$ with the topology induced by the inclusion $M^{(n)} \subset M^n$. It is clear that for each $n\in \N$, we obtain a continuous $G$-action on $M^{(n)}$ by letting $G$ act componentwise. The action of $G$ on $M$ is called 
 \begin{itemize}
 	\item \textbf{proper} if for every compact space $K \subset M$, the set $\{\gamma \in G \mid \gamma K \cap K \neq \emptyset\}$ has compact closure.
 	\item \textbf{cocompact} if there is a compact subset $K \subset M$ such that $M= G K$. 
 	\item \textbf{$n$-proper} (resp. \textbf{$n$-cocompact}) if the induced action on $M^{(n)}$ is proper (resp. cocompact).
 \end{itemize}
We record here some essential features of locally compact hyperbolic groups. Recall from the introduction that a locally compact group $G$ acting continuously on a compact space $M$ is called a \textbf{convergence group} if the action is $3$-proper. The convergence group is called \textbf{uniform} if the action is also $3$-cocompact. A locally compact group is called \textbf{hyperbolic} if it is compactly generated, and if the Cayley graph with respect to some compact generating set is hyperbolic. In fact, hyperbolicity does not depend on the compact generating set, as two compact generating sets will give rise to quasi-isometric Cayley graphs. Note that Cayley graphs of non-discrete groups with respect to compact generating sets have two downsides: they are not proper metric spaces (indeed locally infinite graphs), and the $G$-action is not continuous. However, nicer isometric actions can be produced for any compactly generated locally compact group (see \cite[Proposition 2.1]{Caprace} and references therein).
\begin{prop} \label{prop:nice_space} Let $G$ be a compactly generated locally compact group. Then $G$ acts continuously, properly and cocompactly by isometries on a proper geodesic metric space $X$. If moreover the connected component of the identity is compact, then $X$ can be chosen to be a locally finite graph.
\end{prop} 
In particular, a locally compact group $G$ is hyperbolic if and only if $G$ has a continuous proper cocompact isometric action on a proper geodesic hyperbolic metric space $X$ (see \cite{Caprace}, Corollary $2.6$). Thus, the quotient of $G$ by the compact kernel of the action is a closed cocompact subgroup of $\Isom(X)$. Conversely, $\Isom(X)$ is naturally a locally compact topological group, and if the action of $\Isom(X)$ on $X$ is cocompact, then $\Isom(X)$ is hyperbolic.
\begin{example} \begin{enumerate}
	\item Let $X = \HH^n$ and $G = \Isom(X)$. Then $X$ is a proper geodesic hyperbolic metric space, and $G$ acts transitively on $X$ so $G$ is hyperbolic. Remark moreover that $G$ acts transitively on $\partial X = S^{n-1}$.
	\item Let $X$ and $G$ be as above. Let $\xi$ be a point in the boundary of $\HH^n$. Then the group $H = \RR^{n-1} \rtimes \RR$ (where the right hand factor acts on the left by homothety, i.e. $e^\lambda \Id$) acts transitively on $X$ and fixes $\xi$ ($\RR^{n-1}$ stabilizes each horoball centered at $\xi$, and $\lambda$ is realized as a translation of length $\lambda$ along an axis with $\xi$ as one of its endpoints). Thus $H$ is a cocompact closed subgroup of $G$, and as such is hyperbolic, with the same boundary $S^{n-1}$. Remark that $H$ fixes $\xi$ and acts transitively on $S^{n-1} \backslash \{\xi\}$.
	\item Let $X$ be a locally finite tree such that $G = \Isom(X)$ acts cocompactly on $X$. Then $G$ is a hyperbolic group with a totally disconnected boundary.
	\end{enumerate}
\end{example}

This space $X$ is determined up to quasi-isometry and will serve as the analogue of the Cayley graph in the finitely generated setting. We can define the hyperbolic boundary of a locally compact hyperbolic group as the boundary of such a space $X$. For this reason, many results that hold for discrete hyperbolic groups also hold for locally compact hyperbolic groups. For example, if $G$ is a non-elementary hyperbolic group, then $\partial G$ is compact, perfect and metrizable. Moreover, a locally compact hyperbolic group acts $3$-properly and $3$-cocompactly on its boundary (see \cite{Bo4} for a proof in the discrete setting, the generalization to locally compact groups is straightforward).

The goal of this section is to prove that conversely, uniform convergence groups on compact perfect sets $M$ are hyperbolic and $M$ is $G$-equivariantly homeomorphic to the hyperbolic boundary.

\begin{theoremBowditch} Let $G$ be a locally compact group acting continuously, 3-properly and 3-cocompactly by homeomorphism on a perfect compactum $M$. Then $G$ is hyperbolic, and there is a $G$-equivariant homeomorphism $\partial G \to M$.
\end{theoremBowditch}
The ideas in \cite{Bowditch} valid in the discrete case require modification. Our initial goal will be the construction of an annulus system satisfying the conditions of Theorem \ref{theorem:annulitocrossratios}. One of the key ingredients in this construction is a dynamical reformulation of the convergence property.

\begin{notation}
Throughout this section, $M$ will be a perfect compactum and $G$ will be a locally compact group acting continuously by homeomorphisms on $M$. Note that we do not assume metrizability of $M$ and so we need to work with nets instead of sequences.
\end{notation}
\subsection{A dynamical reformulation of the convergence property} \label{sec:dynamicconvergenceproperty}

In order to reformulate the convergence property in dynamical terms, we first introduce some terminology. We only recall the very basic facts about nets and refer the reader to Section 1 of \cite{Bo4} for a more detailed account. If the reader is unfamiliar with nets, then he could in a first reading skip this paragraph, assume $M$ to be metrizable and replace the word {\em net} with the word {\em sequence} in all that follows.

Recall that a net in $M$ is a map $f:(I,\leq)\rightarrow M$ where $(I,\leq)$ is a directed set. Often, we abbreviate $f(i)$ as $f_i$ for $i\in I$ and we write $f$ as $(f_i)_{i\in I}$ or $(f_i)_i$ for short. We tacitly assume that the domains of two nets $f$ and $g$, when denoted with the same subscripts (i.e. as $(f_i)_i, (g_i)_i$), are the same. We say that a property is true for all {\bf sufficiently large $i$} if there is some $j\in I$ such that the property holds for all $j\leq i\in I$. Given another directed set $(J,\leq)$, we say that a map $J\rightarrow I, j\mapsto i(j)$ is {\bf cofinal} if for any $l\in I$, we have $i(j)\geq l$ for all sufficiently large $j$.
\begin{defn}
\begin{enumerate}
\item
A net $(f_i)_i$ is called {\bf wandering} if given any compact set $K\subset M$, we have $f_i\in M\backslash K$ for all $i$ sufficiently large.
\item
A net $(f_i)_i$ in a set $M$ is said to {\bf converge} to $x\in M$ if given any neighbourhood $U$ of $x$, the elements $f_i$ lie in $U$ for all sufficiently large $i$. 
\item
A {\bf subnet} of $f$ is a net that is obtained by precomposing $f$ with a cofinal map. Given two nets $(f_i)_i, (g_i)_i$, then subnets $(\tilde{f}_j)_j, (\tilde{g}_j)_j$ (i.e. subnets obtained by precomposing with the same cofinal map) are called {\em common subnets} of $f$ and $g$.
\end{enumerate}
\end{defn}
It is interesting to note that, although sequences are special examples of nets, a subnet of a sequence need not be a subsequence. We also state the well known fact that any net in a compact set has a convergent subnet (see for example \cite{Murdeshwar}). Recall that the corresponding statement for sequences, i.e. that every sequence has a convergent subsequence, is not true in general, although it is true in compact {\em metrizable} sets.

\begin{defn}
A net $(\gamma_i)_i\subset G$ is called {\bf collapsing} 
if there are $a,c\in M$ such that $(\gamma_i)_{\mid M\backslash \{a\}}$ converges locally uniformly to $c$.
\end{defn}
\begin{prop} Let $G$ be a locally compact group acting continuously on a compact space $M$. Then $G$ acts as a convergence group if and only if every wandering net has a collapsing subnet.
\end{prop}

The equivalence of both definitions of convergence groups is shown in Section $1$ of \cite{Bo4} in the case of {\em discrete} groups. Up to minor modifications, one can generalize these proofs to the locally compact setting. We remark that when $M$ is also metrizable, then $G$ acts as a convergence group on $M$ if and only if every wandering {\em sequence} of $G$ has a convergent {\em subsequence} (see \cite{Bo4}, \cite{Bowditch}).

\subsection{Constructing a suitable annulus system as in Theorem \ref{theorem:annulitocrossratios}} \label{sec:annulussystemUCG}
The action of $G$ on $M$ induces naturally an action on the closed subsets of $M$ by setting $\gamma C=\{\gamma c\mid c\in C\}$ for every $\gamma \in G$ and $C\subset M$ closed. Consequently, we obtain an action on the collection of annuli of $M$ by setting $\gamma(A^-,A^+)=(\gamma A^-,\gamma A^+)$.
\begin{defn}
An annulus system $\mathcal{A}$ of $M$ is called {\em $G$-invariant} if for any $\gamma \in G$, we have $\gamma A\in \mathcal{A} \Leftrightarrow A\in \mathcal{A}$.
\end{defn}
Assume that $\mathcal{A}$ is a $G$-invariant annulus system such that $\mathcal A / G$ is finite. Let $\mathcal A^* = \{A^*_i\}_{1 \leq i \leq m}$ be a family of representatives of $G$-orbits
. Call a family $\mathcal B \subset \mathcal A$ of annuli \textbf{bounded} if there exists a compact set $K \subset G$ such that $\mathcal B \subset K \mathcal A^*$. This definition does not depend on the choice of $\mathcal A^*$. Clearly, a finite family of annuli is always bounded, but the converse need not hold if $G$ is not discrete. 
	
	\begin{lemma} \label{lem:unbddannuli} Let $\mathcal B$ be a bounded family of annuli. Then there is some $n \in \NN$ such that if $A_1,A_2,\ldots ,A_q\in \mathcal{B}$ with $A_1 < A_2 < \ldots < A_q$, then we must have $q < n$. In other words, there is a bound on the size of a nested subfamily of annuli.
	\end{lemma}
	\begin{proof} Given any annulus $A$, the set $M \backslash (A^+ \cup A^-)$ is nonempty and open so that there is an open neighborhood $U_A$ of the identity in $G$ such that for each $\gamma \in U_A$ one has that $\gamma.A \not < A$. 
	Pick $V_A$ a smaller open neighborhood of the identity such that $V_A^{-1} V_A \subset U_A$. Thus for any $\gamma, \gamma' \in V_A,$ one has $\gamma.A \not < \gamma' .A$. 
	The same holds if $\gamma, \gamma'$ lie in the same left coset of $V_A$.
	
	Let $\{A^*_i\}_{1\leq i \leq m}$ be a family of representatives of $G$-orbits of annuli, and set $V = \cap_{1 \leq i \leq m} V_{A^*_i}$ where $V_{A^*_i}$ is constructed as above. Then $V$ is an open neighborhood of the identity. Since $\mathcal B$ is bounded, there is a compact set $K \subset G$ such that $\mathcal B \subset K \mathcal A^*$. The set $K$ is covered by finitely many $G$-translates $\eta_1 V, \ldots, \eta_p V$ of $V$. Thus any $A \in \mathcal B$ can be written as $A = \eta_{j} \gamma A^*_{i}$ for some $\gamma \in V$.
	
	We claim that we can take $n = mp$. Indeed, suppose that $A_1 < A_2 < \ldots < A_q$ with $A_k \in \mathcal B$. Write $A_k = \eta_{j_k} \gamma_k A^*_{i_k}$ where $\gamma_k \in V$. Suppose $q > mp$ then there exists $1 \leq k < l \leq q$ with $j_k = j_l$ and $i_k = i_l$. But since $\gamma_k$ and $\gamma_l$ are in $V$ this contradicts the fact that $A_k < A_l$.
	\end{proof}
	
	\begin{lemma} \label{lem:3ptsmax} Let $(x_i)_i, (y_i)_i, (z_i)_i, (w_i)_i\in  M$ be sequences. Suppose that we have a sequence of annuli $(A_i)_i$, such that $\{x_i, y_i\} < A_i < \{z_i, w_i\}$ for all $i$ and such that the set $\{A_i\}_i$ is unbounded. Then either $(x_i)_i$ and $(y_i)_i$ have a common subnet converging to a point $x$ or $(z_i)_i$ and $(w_i)_i$ have a common subnet converging to a point $z$.
	\end{lemma}
	\begin{proof} Since $\mathcal A / G$ is finite, we can assume after passing to a subsequence that there is some $A \in \mathcal A$ such that $A_i = \gamma_i A$ and such that the set $(\gamma_i)_i$ is not contained in a compact. After extracting a collapsing subnet, there are $a,c \in M$ such that $\gamma_j|_{M\backslash \{a\}} \to c$ locally uniformly. Without loss of generality (interchanging $\{x_i,y_i\}$ and $\{z_i,w_i\}$ and $A_i$ with $-A_i$ if necessary) we can suppose that $a \notin A^+$. Thus $\gamma_j|_{A^+}$ converges uniformly to $c$, so that for any neighborhood $U$ of $c$ we have $A_j^+ = \gamma_j A^+ \subset U$ for all sufficiently large $j$. But $A_j < \{z_j, w_j\}$, so $z_j, w_j \in A_j^+ \subset U$. Thus $z_j \to c$ and $w_j \to c$.
	\end{proof}
	
	\begin{lemma}\label{lem:infinitexratio}  Let $(x_i)_i, (y_i)_i, (z_i)_i, (w_i)_i\in  M$ be sequences and suppose that\\
$(x_i y_i | z_i w_i)_i \to \infty$. Then up to a subsequence, either $(x_i)_i$ and $(y_i)_i$ have a common subnet converging to a point $x\in M$ or $(z_i)_i$ and $(w_i)_i$ have a common subnet converging to a point $z\in M$.
	\end{lemma}
	\begin{proof} After passing to a subsequence, we can ensure that $(x_i y_i | z_i w_i) \geq i$ for all $i$, so there are annuli $(A_{ij})_{i \in \NN, 1\leq j\leq i}$ such that for each $i$
	\[ \{x_i,y_i\} < A_{i 1} < A_{i 2} < \ldots < A_{i i} < \{z_i, w_i\}. \]
	
	The family $\{A_{ij}\}_{j \leq i}$ contains arbitrarily large nested subfamilies so it is unbounded by Lemma \ref{lem:unbddannuli}. We can extract an unbounded sequence $(A_{j})_{j \in \NN}$ of annuli such that $\{x_j, y_j\} < A_j < \{z_j, w_j\}$ for each $j$. Thus we reduced the problem to Lemma \ref{lem:3ptsmax}.
	\end{proof}
	
	An immediate consequence of Lemma~\ref{lem:infinitexratio} is that  $(xy|zw) < \infty$ for all $(x,y,z,w) \in M^{(4)}$, or in other words that axiom (A1) of Theorem \ref{theorem:annulitocrossratios} is satisfied. Axiom (A2) also follows from Lemma~\ref{lem:infinitexratio}. In fact, as we now show, the proof of \cite[Lemma 7.4]{Bowditch} can be followed with a slight adaptation to this setting.
\begin{lemma}
There is some $k\geq 0$ such that if $(x,y,z,w)\in M^{(4)}$ with $(xy\mid zw)\geq k$, then $(xz\mid yw)=0$.
\end{lemma}
\begin{proof}
We give a proof by contradiction. Assume thus that there is a sequence $(x_i,y_i,z_i,w_i)_i\subset M^{(4)}$ such that $(x_iy_i\mid z_iw_i)\to \infty$ and such that $(x_iz_i\mid y_iw_i)>0$ for all $i$. Then for all $i$, we can find an annulus $A_i$ with $\{x_i,z_i\}<A_i<\{y_i,w_i\}$. Since $\mathcal{A}/G$ is finite (and up to taking a subsequence), we can suppose that each $A_i$ is in the orbit of one same $A\in \mathcal{A}$. Now, by $G$-invariance of the crossratio, we see that we could have chosen the $(x_i,y_i,z_i,w_i)$ such that $\{x_i,z_i\}<A<\{y_i,w_i\}$ for all $i$. In particular, the $w_i$ (and the $y_i$) lie in int$(A^+)$ and the $z_i$ (and the $x_i$ respectively) lie in int$(A^-)$ and moreover $A^-$ and $A^+$ are disjoint closed sets. By Lemma \ref{lem:infinitexratio} however, we can suppose that modulo a common subnet, $z_j\to z$ and $w_j\to z$ for some $z\in M$ (or $x_j\to z$ and $y_j\to z$ respectively). We thus obtain a contradiction.
\end{proof}

Thus $(..|..)$ is a hyperbolic path cross-ratio by Theorem \ref{theorem:annulitocrossratios}.\\
%
%

Finally, let us assume that $M$ is also perfect and that $G$ is a {\em uniform} convergence group. Let $\Theta_0 \subset M^{(3)}$ be a compact set with $M^{(3)} = G \Theta_0$. We recall Bowditch's construction of an annulus system in this case. Given $\theta = (x, y, z) \in M^{(3)}$, choose open subsets, $U(\theta), V (\theta)$ and $W(\theta)$, of $M$, containing $x, y$, and $z$ respectively, whose closures, $\overline{U(\theta)}, \overline{V(\theta)}$ and $\overline{W(\theta)}$, are pairwise disjoint. Let $\Theta(\theta) = U(\theta)\times V (\theta) \times W(\theta) \subset  M^{(3)}$. We now find a finite set $\theta_1, \ldots, \theta_n \in  \Theta_0$, such that $\Theta_0 \subset \cup_{i=1}^n \Theta(\theta_i)$. Let $A_i$ be the annulus $(\overline{U(\theta_i)}, \overline{V(\theta_i)})$. Let $\mathcal A$ be the set of annuli of the form $\gamma A_i$ or $-\gamma A_i$ as $\gamma$ ranges over $G$ and $i$ runs from $1$ to $n$. Thus, $\mathcal A$ is symmetric and $G$-invariant, and $\mathcal A/G$ is finite. The following lemma shows that this annulus system satisfies condition (A3) of Theorem \ref{theorem:annulitocrossratios}.
\begin{lemma}
If $K\subset M$ is closed and $x\in M\backslash K$, then there is some $A\in \mathcal{A}$ with $K<A<x$. Moreover, condition $(A3)$ is satisfied.
\end{lemma}
\begin{proof}
Choose $y\in M\backslash \{x\}$ and a net $(x_i)_i\subset M\backslash \{x,y\}$ converging to $x$ in the original topology on $M$ (not in the crossratio topology). By $3$-cocompactness, we can find a 
net $(\gamma_i)_i\subset G$ such that the $\gamma_i(x,y,x_i)$ lie in $\Theta_0\subset M^{(3)}$ and (after passing to a subnet) 
converge to some $(a,b,c)\in \Theta_0$. By construction of $\mathcal{A}$, we find $B\in \mathcal{A}$ such that $b<B<a$. After passing 
to a further collapsing subnet, we find that $(\gamma_i)_{\mid M\backslash \{x\}}$ converges locally uniformly to $b$. So, for 
sufficiently large $i$, $\gamma_iK<B<a$, so $K<A<x$ where $A=\gamma_i^{-1}B\in \mathcal{A}$.

Now, letting $M\backslash \mbox{int}(A^+)$ play the role of $K$, we obtain another annulus $A_1\in \mathcal{A}$ with $M\backslash \mbox{int}(A^+)<A_1<x$, i.e. $K<A<A_1<x$. Continuing in this fashion, we conclude $(K\mid x)=\infty$ so that in particular condition $(A3)$ is satisfied.
\end{proof}
\begin{remark} By Theorem \ref{theorem:annulitocrossratios}, we can conclude that the topology on $M$ coincides with the crossratio topology, so that in particular, the topology on $M$ is metrizable. We are thus allowed to work with sequences instead of nets in further arguments.
\end{remark}
\subsection{Proof of Theorem \ref{theorem:Bowditch}} \label{sec:proofUCG}
So far, we have constructed an annulus system satisfying all of the conditions of Theorem \ref{theorem:annulitocrossratios}. This gives rise to a hyperbolic path-crossratio on $M$, compatible with the topology. Thus the space $Q = M^{(3)}$ is naturally endowed with a $G$-invariant quasimetric $\rho$. Theorem~\ref{theorem:crossratiotoquasimetric} implies that $(Q,\rho)$ is a hyperbolic path quasi-metric space, and that $M$ is naturally (and hence $G$-equivariantly) homeomorphic to the hyperbolic boundary $\partial Q$.

It now remains to be shown that $G$ is $G$-equivariantly quasi-isometric to $(Q,\rho)$. As $G$ is a uniform convergence group, it acts properly and cocompactly on the quasi-metric space $(Q,\rho)$. So, to finish the proof, we will formulate our \emph{ad-hoc} adaptation of the Schwartz-Milnor Lemma to the setting of quasi-metric spaces.
	\begin{lemma}[Schwartz-Milnor] \label{lemma:SchwartzMilnor} Let $Q$ be a topological space equipped with a path quasi-metric $\rho$ such that $\rho$-bounded sets in $Q$ coincide with subsets having compact closure. Suppose that $G$ is a locally compact group acting continuously, properly and cocompactly by homeomorphisms and by isometries on $Q$. Then $G$ is compactly generated and any orbit map is a quasi-isometry from $G$ to $Q$.
	\end{lemma}
\begin{proof}
The lemma can easily be derived by adapting the arguments of the classical Schwartz-Milnor lemma in the context of metric spaces (see for example \cite{BridHaef}). 
\end{proof}
The proof of Theorem \ref{theorem:Bowditch} thus reduces to showing that $\rho$-bounded sets on $Q$ coincide with sets of compact closure. In order to show this, let us elaborate on two natural topologies on $M^{(3)}\cup M$.

The first topology was already described in Subsection \ref{subsc:Quasimetrics}: first, fix $r\in \R^+$ and a 
basepoint $a\in Q=M^{(3)}$. Next, given $x\in M$,  take a ($k$-)geodesic $(X_i)_i$ emanating from $a$ and 
in the class of $x$. Then, $\forall n\in \N$, we define $D(n)=\{y\in Q \cup \partial Q\mid$ a geodesic 
connecting $a$ to $y$ meets $N_\rho(X_n,r)=\{y\in Q \mid \rho(X_n,y)<r\} \}$. We define the $(D(n))_n$ to be 
a base of neighbourhoods of $x$. Putting the discrete topology on $Q=M^{(3)}$, we obtain a well defined 
topology on $M^{(3)}\cup M$, which we will call the {\bf $\rho$-topology}. 

For the second approach, we consider $M^3$ with the product topology coming from the metric topology on $(M,d)$. Next, to any $x\in M$, we associate the set $M_x=\{(y,z,w)\in 
M^3\mid$ at least two of the coordinates $y,z,w$ are equal to $x\}$. 
We will call two triples $(x_1,y_1,z_1)\in M^3$ and $(x_2,y_2,z_2)\in M^3$ {\em equivalent} if they are equal or if they both belong to a same 
set $M_x$. The quotient topology with respect to this equivalence relation gives rise to a topology on $M^{(3)}\cup M$, which we call the {\bf quotient topology}.

\begin{lemma}
Any $\rho$-unbounded sequence $(X_i)_{i\in \N}$ in $M^{(3)}$ contains a subsequence which converges to some $x\in M$. Here, convergence is with respect to the quotient topology.
\label{lemma:unboundedimpliesconvergence}
\end{lemma}
\begin{proof}

Let $(X_i)_i=(x_i^1,x_i^2,x_i^3)_i$ be an unbounded sequence and fix a base point $(a,b,c)\in M^{(3)}$.
Then
\[ \rho((a,b,c),(x_i^1,x_i^2,x_i^3))=\max\{(ab\mid x_i^jx_i^k), (bc\mid x_i^jx_i^k), (ac\mid x_i^jx_i^k)\mid j,k \in \{1,2,3\}, j\neq k \},\]
goes to infinity as $i\to \infty$.
Let us assume without loss of generality (and up to taking a subsequence) that $(ab\mid x_i^1x_i^2)\to \infty$ and that $a,b\notin \{x_i^1,x_i^2 \mid i\in \N\}$.
Applying Lemma \ref{lem:infinitexratio} yields that, up to a subnet, we can assume that $(x_i^1)_i$ and $(x_i^2)_i$ both converge to some $x\in M$. As the topology on $M$ is compatible with the crossratio topology, we know it is metrizable. We can thus replace the word {\em net} by {\em sequence} and conclude that there is a subsequence of $(X_i)_i$ converging to some $x\in M$.
\end{proof}
\begin{corollary}
Equip $M^{(3)}\subset M^3\approx M^{(3)}\cup M$ with the subspace topology coming from the quotient topology. Every compact set $K\subset M^{(3)}$ is $\rho$-bounded.
\label{corollary:compactisrhobounded}
\end{corollary}
It turns out that the converse also holds:
\begin{lemma}
Equip $M^{(3)}\cup M$ with the quotient topology.  Every $\rho$-bounded set $K\subset M^{(3)}$ has compact closure in $M^{(3)}$.
\label{lm:rhoboundedisrelcompact}
\end{lemma}
\begin{proof}
First, we show that any sequence $(X_i)_i$ in $M^{(3)}\subset M^{(3)}\cup M$ converging to some $x\in M$ in the quotient topology, is $\rho$-unbounded.
Consider thus a sequence $(X_i)_i=(x_i^1,x_i^2,x_i^3)_i$ converging to some $x\in M$. Modulo a subsequence, we can assume that 
$(x_i^1)_i$ and $(x_i^2)_i$ both converge to the same element $x\in M$. For any point $(a,b,c)\in M^{(3)}$ with $a\neq x \neq b$, we have 
\[ \rho((a,b,c),(x_i^1,x_i^2,x_i^3))\geq (ab\mid x_i^1x_i^2),\]
where the latter crossratio goes to infinity: indeed, using Theorem \ref{theorem:approxtree}, we can compare any $5$-element subset $\{a,b,x_i^1,x_i^2,x\}$ with a metric tree and we know that $(ab\mid xx_i^1)$ and $(ab\mid xx_i^2)$ go to infinity as the crossratio topology coincides with the topology on $M$. 

So now, if $K\subset M^{(3)}$ is a $\rho$-bounded set, then any convergent sequence in $K$ must converge in $M^{(3)}\subset M^{(3)}\cup M$. We conclude that $K$ has compact closure in $M^{(3)}$.
\end{proof}
	\section{$3$-cocompactness and conical limit points \label{sc:Bowditch+}}
Throughout this section, we assume that $(M,d)$ is a perfect, metrizable compactum and that $G$ is a locally compact group acting continuously by homeomorphisms on $M$. If $G$ is discrete, then both Bowditch and Tukia give a characterization of the $3$-cocompactness of the action in terms of {\em conical limit points} (See Definition \ref{def:conical} below). We aim to generalize this  result to the case of locally compact groups.

Note that we did not use metrizability of $M$ in the proof of Theorem \ref{theorem:Bowditch}. 
It is  unclear whether the metrizability assumption can be omitted in Theorem \ref{Theorem:Bowditch+} as the proof of Proposition \ref{proposition:constructannulisystem} below makes explicit use of the metric on $M$.
\subsection{Conical limit points}
\begin{defn}
Assume that $G$ acts as a convergence group on $M$ and let $a\in M$. We say that $a$ is a {\bf conical limit point} if there are $b\neq c\in M$ and a sequence $(\gamma_i)_i\subset G$ such that $(\gamma_i)_{\mid M\backslash \{a\}}\to \{c\}$ locally uniformly on $M\backslash \{a\}$ and $\gamma_i(a)\to b$. Note that, in particular, $(\gamma_i)_i$ is a collapsing sequence. \label{def:conical}
\end{defn}
Bowditch and Tukia, both with their own approach, use conical limit points to characterize the {\em uniformity} of a convergence group.
\begin{theorem}[Bowditch \cite{Bowditch}, Tukia \cite{Tukia}]
Suppose that $M$ is a perfect, metrizable compactum, and that $\Gamma$ is a discrete group with a convergence action on $M$. Then $\Gamma$ is uniform, i.e. acts $3$-cocompactly, if and only if every point $a\in M$ is a conical limit point.
\label{theorem:BowTuk}
\end{theorem}
In this section, we prove the following.
\begin{theoremBowditch+}
Suppose that $M$ is a perfect, metrizable compactum, and that $G$ is a {\em locally compact} group with a (always continuous) convergence action on $M$. Then $G$ is uniform, i.e. acts $3$-cocompactly, if and only if every point $a\in M$ is a conical limit point.
\label{theorem:Bowditch++}
\end{theoremBowditch+}
We immediately obtain
\begin{corollary}
A locally compact hyperbolic group can be characterized as a group acting continuously and $3$-properly on a perfect metrizable compactum $M$ such that every $a\in M$ is a conical limit point.
\end{corollary}

The hard part to Theorem $A_2$ is to show that $3$-cocompactness is implied by the fact that every point $a\in M$ is a conical limit point. The other direction follows immediately from Lemma \ref{lemma:conicallimiteqdef} below (See Tukia's paper \cite{Tukia}, bottom of page 5). Actually, Tukia's proof only considers the case of discrete groups, but the locally compact case can be proved via minor modifications.
\begin{lemma}[\cite{Tukia}]
Assume that the locally compact group $G$ admits a convergence action on a compact, perfect, metrizable set $M$.
A point $x\in M$ is a conical limit point if and only if $\forall z\in M\backslash \{x\}$, there exists a compact set $K\subset M^{(3)}$ and a sequence $(x_i)_i\subset M\backslash \{x,z\}$ converging to $x$ such that $(x,z,x_i)\in G K$.
\label{lemma:conicallimiteqdef}
\end{lemma}

Similar to the proof of Theorem $A_1$, the proof of Theorem $A_2$ will also require the construction of a specific annulus system. We start by introducing a new property $(A4)$ of an annulus system which implies property $(A3)$ if every point of $M$ is a conical limit point.
\begin{lemma}[cfr Lemma $8.3$ in \cite{Bowditch} ]
Suppose that $G$ acts as a convergence group on $M$ and that $\mathcal{A}$ is a symmetric $G$-invariant annulus system satisfying the following condition:
\[ (A4): \mbox{ If } x,y\in M \mbox{ are distinct, then } (x\mid y)>0 .\]
If $x\in M$ is a conical limit point, then $(K\mid x)=\infty$ for every compact subset $K\subseteq M\backslash \{x\}$. In particular, if every $x\in M$ is a conical limit point, then $\mathcal{A}$ satisfies condition $(A3)$ of Theorem \ref{theorem:annulitocrossratios} (i.e. $(x\mid yz)=\infty$ for all $x,y,z\in M$ distinct).
\label{lm:seperateconicalcompact}
\end{lemma}
\begin{proof}
Let $x\in X$ be a conical limit point and let $(\gamma_i)_i,b$ and $c$ be as in the definition of conical limit point.  As $b\neq c$, condition $(A4)$ implies $(b\mid c)>0$ and so there is an annulus $A\in \mathcal{A}$ seperating $b$ and $c$. As $(\gamma_i)\to c$ uniformly on compact sets of $M\backslash \{x\}$, we see that for any compact $K\subseteq M\backslash \{x\}$, there is $i$ large enough such that $\gamma_i K < A < \gamma_i x$ and so   $K<\gamma_i^{-1}A<x$. We conclude that $(K\mid x)>0$.

Denote $A_1=\gamma_i^{-1}A$ and note that $x\in \mbox{int}(A_1^+)$. We can follow the same reasoning as above where the role of $K$ is now played by $M\backslash \mbox{int}(A_1^+)$. This way, we find $A_2\in \mathcal{A}$ such that $K<A_1<A_2<x$. Continuing inductively in this manner, we conclude that $(K\mid x)=\infty$.
\end{proof}
\subsection{The proof of Theorem \ref{Theorem:Bowditch+}}
We start by constructing an annulus system om $M$ which satisfies conditions $(A1),(A2),(A4)$ (and thus $(A1),(A2),(A3)$). In Section \ref{sc:Bowditch}, we were able to use the $3$-cocompactness for this, but here, we will need to rely on the metrizability of $M$. In order to construct the annulus system, let us explain how to generalize Bowditch's Proposition $8.2$ (\cite{Bowditch}) to the setting of locally compact groups. The key to this is Lemma \ref{lemma:smalldiameter} below. We start with some notations.

Given an annulus $A=(A^-,A^+)$, we write $\lambda(A):=\min\{\diam(A^-),\diam(A^+)\}$ where diam stands for the diameter of a set in $(M,d)$. We moreover define $\mu(A):=d(A^-,A^+)$.
\begin{lemma}
Let $s>0$ and let $K\subset G$ be compact. Given any $x\neq y$, one can find an annulus $A$ in $(M,d)$ seperating $x$ and $y$ such that $\lambda(\gamma A)<s$ for every $\gamma \in K$.
\label{lemma:smalldiameter}
\end{lemma}
\begin{proof}
Given $x\in M$, denote the metric ball with center $x$ and radius $r$ by $B(x,r)$. Let $C\subset M$ be a compact set and fix $\epsilon>0$. We show first that there is an open neighbourhood $V\subset G$ of the identity such that $d(\nu c, c)<\epsilon$ for every $\nu\in V, c\in C$. In particular, $\diam(\nu C)<\diam(C)+2\epsilon$ for every $\nu\in V$.

For every $c\in C$, there exists a neighbourhood $U_c\subset M$ of $c$ and a neighbourhood $V_c\subset G$ of the identity such that $V_cU_c\subset B(c,\epsilon/2)$. Clearly, as $1\in V_c$, we get that $U_c\subset B(c,\epsilon/2)$. Cover $C$ with all these $U_c$ and consider a finite subcover $\{U_{c_i}\mid i=1,2,\ldots ,n\}$. Denote $V_C=\cap_{i=1}^n V_{c_i}$. Now, given $c\in C$, choose $c_i\in C$ such that $c\in U_{c_i}$. Then for all $\nu \in V_C$, we have $d(\nu c, c)\leq d(\nu c, c_i) + d(c_i,c)< \epsilon/2+\epsilon/2 =\epsilon$ as desired.

Now, take an annulus $B$ seperating $x$ and $y$ such that $\lambda(B)<s/2$. As $B^-\cup B^+$ is compact, we can define $V_B=V_{C}$ as above where $C=B^-\cup B^+$ and $\epsilon=s/4$. For each $\gamma\in K, \gamma B$ is also an annulus and so we obtain a collection of opens $\{V_{\gamma B}\}_{\gamma \in K}$. Clearly, the $(V_{\gamma B} \cdot \gamma)_{\gamma \in K}$ cover $K$, so we can derive a finite subcover $\{V_{\gamma_i B} \gamma_i\}_{i=1,2,\ldots ,m}$.

Now, take closed subsets of $B^-,B^+$ to obtain another annulus $A$ seperating $x$ and $y$ such that additionally $\lambda(\gamma_i A)<s/2$ for $i=1,2,\ldots m$. Given any $\gamma \in K$, we can write it as $v_i\gamma_i$ for some $i=1,2,\ldots ,m$ and some $\nu_i\in V_{\gamma_i B}$. Consequently, $\lambda(\gamma A)=\lambda(\nu_i\gamma_iA)<\lambda(\gamma_iA)+2\epsilon<s/2+s/2=s$, as desired.
\end{proof}

For every $n\in \N$, we define the set $\Pi(n)=\{(x,y)\in M^{(2)} \mid d(x,y)\geq 1/n\}$. It is clear that $\Pi(n)$ is compact as a closed subset of $M\times M$, so the $\{\Pi(n)\}_n$ form a compact exhaustion of the space $\Pi:=M^{(2)}$.
\begin{prop}
Suppose that a locally compact group $G$ acts as a convergence group on a perfect, metrizable, compact set $(M,d)$. Then, there exists a symmetric $G$-invariant annulus system $\mathcal{A}$ on $M$ such that if $(x,y,z,w)\in M^{(4)}$, then the three quantities $(xy\mid zw), (xz,yw)$ and $(xw\mid yz)$ are all finite and at least two of them are equal to $0$. Moreover, if $x,y\in M$ are distinct, then $(x\mid y)>0$. This annulus system thus satisfies the conditions $(A1),(A2),(A4)$.
\label{proposition:constructannulisystem}
\end{prop}
\begin{proof}
Following the proof of Proposition $8.2$ in \cite{Bowditch}, we will inductively construct a sequence of symmetric $G$-invariant annulus systems $\mathcal{A}(n)$ with $\mathcal{A}(n)/G$ finite.  Fix $n\in \N$. Writing $(K\mid L)_n:=(K\mid L)_{\mathcal{A}_n}$ for short, we assume by induction that for all $(x,y,z,w)\in M^{(4)}$, at least two of the quantities $(xy\mid zw)_n, (xz\mid yw)_n$ and $(xw\mid yz)_n$ are equal to $0$, that $(x\mid y)_n>0$ for every $(x,y)\in \Pi(n)$ and that $\lambda(A), \mu(A)>0$ for every $A\in \mathcal{A}(n)$. Let us now construct a symmetric $G$-invariant annulus system $\mathcal{A}(n+1)$ which satisfies the above conditions for $n$ replaced by $n+1$.

Let $\mu$ be the minimal value of $\{\sup_{\gamma \in G}(\mu(\gamma A))\}$, where $A$ ranges over $\mathcal{A}(n)$. Note that $\mu>0$ because $\mathcal{A}(n)/G$ is finite. Using the convergence group hypothesis, we see that given any annulus $A$ and $\epsilon>0$, there is a compact set $K\subset G$ such that $\lambda(\gamma A)<\epsilon$ for all $\gamma \in M\backslash K$. Using this, together with Lemma \ref{lemma:smalldiameter}, we see that given any $\pi=(x,y)\in \Pi(n+1)$, we can find an annulus $A(\pi)=(A^-(\pi),A^+(\pi))$ seperating $x$ and $y$ such that $\lambda(\gamma A(\pi))<\min(\mu/2, \frac{1}{n+2}\}$ for all $\gamma \in G$. Moreover, because $M$ is a perfect metric space and since $d(x,y)\geq \frac{1}{n+1}$, one can make sure that additionally, $\mu(A)>\frac{1}{n+2}$.

The $(\mbox{int}(A(\pi)^-)\times \mbox{int}(A(\pi)^+))_{\pi \in \Pi(n+1)}$ cover the compact set $\Pi(n+1)$ and so we can derive a finite subcover. The corresponding elements $\pi$ of the subcover form a finite set $\{\pi_1,\pi_2,\ldots ,\pi_p\}\subseteq \Pi(n+1)$ for some $p\in \N$.
Let $\mathcal{B}=\cup\{\gamma A(\pi_i),-\gamma A(\pi_i)\mid 1\leq i \leq p, \gamma \in G\}$. For every $A\in \mathcal{B}$, we have 
$\lambda(A)<\min(\mu/2, 1/(n+2))$. Define $\mathcal{A}(n+1)=\mathcal{A}(n)\cup \mathcal{B}$. Clearly, this is a symmetric $G$-invariant 
annulus system and by construction $(x\mid y)_{n+1}>0$ for every $(x,y)\in \Pi(n+1)$.

Assume by contradiction that for some $(x,y,z,w)\in M^{(4)}$, we have $(xy\mid zw)_{n+1}>0$ {\em and} $(xz\mid yw)_{n+1}>0$. Then we 
can take $A,B\in \mathcal{A}(n+1)$ such that $\{x,y\}<A<\{z,w\}$ and $\{x,z\}<B<\{y,w\}$. We will now show that 
$\lambda(A)<\mu(B)$, which will immediately give the desired contradiction as $\mu(B)\leq d(x,y)\leq \lambda(A)$ or $\mu(B)\leq d(z,w)\leq \lambda(A)$.
By induction, we see that not both $A$ and $B$ can lie in $\mathcal{A}(n)$, so we can proceed under the assumption $A\in \mathcal{B}$. By construction, $\lambda(A)<\min(\mu/2, 1/(n+2))$. On the other hand, by $G$-invariance, if $B\in \mathcal{B}$, then we can choose another point in the $G$-orbit of $(x,y,z,w)$ if necessary to ensure that for the corresponding $B$, $\mu(B)>1/(n+2)$. Similarly, if $B\in \mathcal{A}(n)$ then we can arrange that $\mu(B)>\mu/2$. Consequently, $\lambda(A)<\mu(B)$ as desired.

Thus, starting the induction process from $\mathcal{A}(0)=\phi$, we obtain a symmetric $G$-invariant annulus system $\mathcal{A}:=\cup_{n\in \N} \mathcal{A}(n)$. We denote the associated crossratio by $(..\mid..)$. Since $\Pi=\cup_{n\in \N} \Pi(n)$, we see that $(x\mid y)>0$ for all $(x,y)\in M^{(2)}$. 

Finally, let us verify that $(xy\mid zw)<\infty$ for every $(x,y,z,w)\in M^{(4)}$. Denote $\lambda=\min(d(x,y),d(z,w))$ and choose $n\in \N$ with $1/n<\lambda$. If $\{x,y\}<A<\{z,w\}$, then $\lambda(A)\geq \lambda>1/n$, so $A\in \mathcal{A}(n)$. The convergence group hypothesis together with Lemma \ref{lem:unbddannuli} show that there is a bound on the size of a nested sequence of annuli in $\mathcal{A}(n)$ that seperate $\{x,y\}$ and $\{z,w\}$.
\end{proof}

As an immediate corollary of the construction above, Lemma \ref{lm:seperateconicalcompact}, Theorem \ref{theorem:annulitocrossratios} and Theorem \ref{theorem:crossratiotoquasimetric}, we obtain
\begin{corollary}
Assume that a locally compact group $G$ acts $3$-properly on a compact, perfect, metrizable set $(M,d)$ such that every point of $M$ is a conical limit point. Then the annulus system constructed in Proposition \ref{proposition:constructannulisystem} induces a hyperbolic path crossratio $(..\mid..)$ on $M$ which is compatible with the topology.
This crossratio induces a $G$-invariant hyperbolic path quasimetric $\rho$ on the space $Q(=M^{(3)})$ of distinct triples of $M$. Moreover, the hyperbolic boundary $\partial Q$ can be $G$-equivariantly identified by a homeomorphism with $M$ and the crossratio $(..\mid..)_\rho$ on $\partial Q=M$, induced by the crossratio associated to $\rho$ on $Q$, differs from $(..\mid..)$ by at most an additive constant.
\label{corollary:machinery}
\end{corollary}
With the machinery of Corollary \ref{corollary:machinery} in place, we require one last lemma before we can finish the proof of Theorem \ref{Theorem:Bowditch+}.
\begin{lemma} \label{lm:rhounbounded}
Any $\rho$-unbounded sequence $(X_i)_{i\in \N}$ in $Q\equiv M^{(3)}$ contains a subsequence which converges, with respect to the $\rho$-topology, to some $x\in M$.
\end{lemma}
\begin{proof}
Let $(X_i)_i$ be an unbounded sequence and for each $X_i$, let $\gamma_i$ denote a $k$-geodesic connecting some basepoint $(a,b,c)$ to $X_i$. Recall that by definition, $\gamma_i$ is a set of points $(a,b,c)=x_0^i, x_1^i,x_2^i, \ldots, x_{n(i)}^i=X_i$ such that $\rho(x_r^i,x_s^i)\aeq_k \lvert r-s\rvert$ for all $r,s\in \{1,2,\ldots ,n(i)\}$. As $\rho$-bounded balls are compact, the sequence $(x_1^i)_i$ contains a subsequence $(x_1^j)_j$ converging to some $x_1\in Q$. Similarly, the sequence $(x_2^j)_j$ contains a subsequence converging to some $x_2\in Q$ and so on. It is easy to check that the sequence $(x_l)_l$ is a $3k$-geodesic ray and so it corresponds to an element $x\in \partial Q\equiv M$. Now, for every $k\in \N$, there is some $m(k)$ such that the path $(x_1,x_2,\ldots ,x_k)$ stays at a bounded distance from a $k$-geodesic path connecting $(a,b,c)$ to $X_{m(k)}$. This subsequence of $(X_i)_i$ then $\rho$-converges to $x$ by definition. 
\end{proof}

\begin{proof}[Proof of Theorem \ref{Theorem:Bowditch+}]
Fix $\theta \in Q$ and assume by contradiction that the action is not $3$-cocompact. As $\rho$-bounded sets in $M^{(3)}$ have  compact closure (see Lemma \ref{lm:rhoboundedisrelcompact}) and since $\rho$ is $G$-invariant, there exists a sequence $(u_i)_i$ in $Q$ with $\rho(u_i,G \theta)\to \infty$. The $G$-invariance of $\rho$ implies that we can suppose that $\theta$ is a nearest point of $G \theta$ to $u_i$ for all $i$. As $(u_i)_i$ is $\rho$-unbounded, it contains a subsequence $\rho$-converging to some $x\in M$ (see Lemma \ref{lm:rhounbounded}). It is easy to see that if $(v_i)_i$ is another sequence in $Q$ that $\rho$-converges to $x$ and stays at bounded distance from some geodesic ray, then $\rho(v_i,G \theta)$ must go to infinity.

Since $x$ is a conical limit point, Lemma \ref{lemma:conicallimiteqdef} shows the existence of a compact set $\Theta_0\subseteq M^{(3)}$, a point $y\in M\backslash \{x\}$ and a sequence $x_i\to x$ such that $w_i=(y,x,x_i)\in G \Theta_0$. The points $(w_i)_i$ lie at bounded distance from a geodesic ray connecting $(y,x,x_0)$ to $x$: indeed, a typical geodesic ray emanating from $(y,x,x_0)$ in the class of $x\in \partial Q$ is obtained from interpolating between the points $(y,x,x_i)$ (see the proof of Lemma \ref{lm:infinitegeodesic}). In particular, $(w_i)_i$ $\rho$-converges to $x$. However, since $\Theta_0$ is compact, it is $\rho$-bounded by Corollary \ref{corollary:compactisrhobounded} and so lies in some bounded $\rho$-ball with centre $\theta$. Consequently, $\rho(w_i,G \theta)$ is bounded, a contradiction.	
\end{proof}

\section{Sharply $n$-transitive actions} \label{sc:sharply-n-transitive}
\subsection{Introduction and the case $n\geq 4$}
In this paragraph, we discuss known results about sharply $n$-transitive actions and we report on the status of the classification of continuous, sharply $n$-transitive actions of $\sigma$-compact, locally compact groups on compact spaces. The classification is still open for $n=2$ and $n=3$.  
%

Sharply $n$-transitive actions exist for any $n$: the symmetric group $S_n$ acts both sharply $n$-transitively and sharply $(n-1)$-transitively on the set $\{1,2,\ldots ,n\}$. Further, it can be shown that the alternating group $A_n$ acts sharply $(n-2)$-transitively. In $1872$, Jordan \cite{Jordan} showed the following result.
\begin{theorem}[see \cite{Jordan}]
For $n\geq 6$, the only finite sharply $n$-transitive groups are the trivial ones, i.e. $S_n,S_{n+1},A_{n+2}$.
For both the cases $n=4$ and $n=5$ there is one additional group: the Mathieu group $M_{11}$ 
for $n=4$ and the Mathieu group $M_{12}$ 
for $n=5$.
\end{theorem}
Later, J. Tits (\cite{Tits1} (1952))  generalized this result to the infinite case by showing that there are no infinite groups acting sharply $n$-transitively on any set for $n\geq 4$. In particular, this reduces the study to the cases $n = 2$ and $n=3$. 


\subsection{The classification for $n=2$ \label{subsc:n=2}}
In \cite{Tits3} $(1952)$, Tits classifies all sharply $2$-transitive actions of $\sigma$-compact, locally compact groups on locally compact, connected, first countable topological spaces. 
As we will show, if we require the action to be on a (possibly not connected, possibly not first countable) {\em compact} space $M$, then we can reduce to the finite case.
Let us first record a standard fact (see the corollary to Theorem $8$ in \cite{Are46} or Ch VII, App. 1, Lemme 2 in \cite{Bou63}).
\begin{lemma} \label{lemma:Baireargument}
	Let $G$ be a $\sigma$-compact, locally compact group acting continuously and transitively on a locally compact space $X$ and let $x \in X$. Then the orbital map $f : G / G_x \to X$ is a homeomorphism.
\end{lemma}
\begin{proof}[Proof Sketch]
The fact that $f$ is a continuous bijection follows immediately. Using the fact that $G$ is $\sigma$-compact and that $X$ is a Baire space, one can show that $f$ is in fact open, hence a homeomorphism.
\end{proof}
We obtain the following important consequence on sharply $n$-transitive groups.
\begin{lemma} \label{lemma:sharplytransisproper}
Let $G$ be a $\sigma$-compact, locally compact group acting continuously and sharply $n$-transitively on a compact space $M$, and let $x \in M^{(n)}$. Then the orbit map $f:G \to M^{(n)} : \gamma \mapsto \gamma x$ is a homeomorphism. In particular, the action of $G$ on $M$ is $n$-proper.
\end{lemma}
\begin{proof}
As the action is sharply $n$-transitive, the stabilizer of any point $x\in M^{(n)}$, denoted $G_x$, is trivial. So, the orbit map $f : G \to M^{(n)}$ is an (equivariant) homeomorphism. 
As the action of any topological group on itself by left multiplication is proper, it follows that the action of $G$ on $M$ is $n$-proper.
\end{proof}
As a first step towards understanding sharply $2$-transitive actions, we show that $G$ must be compact. This statement appears in a preprint of Yaman \cite{Yaman}.
\begin{lemma}
 \label{lem:P1=P2} If $G$ acts $2$-properly and continuously on a compact space $M$, then $G$ is compact or $M$ is a singleton.
		\end{lemma}
		\begin{proof}
			Arguing by contradiction, suppose that there exists a non-compact group $G$ acting continuously and 2-properly on a compact set $M$ with at least $2$ points. Pick a wandering sequence $(\gamma_n)_n$ and pick $x \neq y \in M$. Because of compactness, the sequence $(\gamma_n(x))_n$ 
			has a limit point $x'$. 
			For any neighbourhood $U$ of $x'$, there is a subsequence $(\alpha_n(x))_n$ of $(\gamma_n(x))_n$ which lies in $U$. Similarly, for $(\alpha_n(y))_n$ there is such a point, say $y'$, and we can take a further subsequence $(\beta_n(y))_n$ which lies in a neighbourhood $V$ of $y'$. If $x'\neq y'$, then we can take $U,V$ to be disjoint compact neighbourhoods which contradicts $2$-properness. We conclude $x'=y'$.

Now, take a neighbourhood $U$ as above and $z'\notin U$. There exists a limit point $z\in M$ 
of $\{ \beta_n^{-1}(z') \mid n\in \N\}$. Up to exchanging $x$ and $y$ we can assume that $z \neq x$. 
Let $W$ be a compact neighbourhood of $z$ not containing $x$ and take a subsequence $(\delta_n)_n\subset (\beta_n)_n$ such that $\delta_n^{-1}(z')\in W$ for all $n$. Letting $z_n = \delta_n^{-1} z'$, we have that $\delta_n z_n=z'$ and $z'\notin U\ni \delta_n(x)$, which contradicts $2$-properness.
		\end{proof}
\begin{prop}
Let $G$ be a $\sigma$-compact, locally compact group acting continuously and sharply $2$-transitively on a compact space $M$, then both $G$ and $M$ are finite.
\label{prop:2-transitiveimpliesfinite}
\end{prop}
\begin{proof}
If $G$ acts sharply $2$-transitively on $M$, then by Lemma \ref{lemma:sharplytransisproper}, the action is $2$-proper and $G\cong M^{(2)}$. Lemma \ref{lem:P1=P2} then implies that $G$ is compact and so $M^{(2)}=M^2\backslash \{(x,x)\mid x\in M\}$ is also compact. In particular, $M^{(2)}$ must be closed in $M^2$ and so the topology on $M$ must be discrete. As $M$ is compact, this implies that $M$ and so $G$ are finite. 
\end{proof}
Zassenhaus classified all finite sharply $2$-transitive groups as the affine transformation groups of finite near fields and he moreover classified all finite near-fields \cite{Zassenhaus1,Zassenhaus2}.
A {\em near-field} is an algebraic structure with two operations, addition and multiplication, satisfying all of the axioms for a field with the possible exception of the commutative law for multiplication $(xy=yx)$ and the left distributive law $(x(y+z)=xy+xz)$. The {\em affine transformations} of a near field $M$ are the maps of the form $x\mapsto ax+b$ where $a,b\in M, a\neq 0$ and they form a sharply $2$-transitive group called the {\em affine transformation group}. Combining his result with ours, we obtain the following 
\begin{prop} 
Let $G$ be a $\sigma$-compact, locally compact group acting continuously and sharply $2$-transitively on a compact set $M$. Then $M$ is a finite near-field and $G$ is the affine transformation group of $M$.
\end{prop}
\subsection{Towards a classification for $n=3$}
The general case of $\sigma$-compact, locally compact groups acting sharply $3$-transitively on compact sets remains open.

A first known subcase is the case where the space $M$, and hence the group $G$, are finite. Zassenhaus showed in \cite{Zassenhaus1} that $M$ must be the projective line $\Projline(\FF_q)$ associated to a finite field $\FF_q$. Here, $q=p^m$ for a certain prime $p$ and $m\in \NN$. When $m$ is odd, the only possible $G$ is the projective group $\PGL_2(\FF_q)$. When $m$ is even, there is also the other possibility of the Mathieu group. We refer to \cite{Zassenhaus1} for details.

The second known case concerns $\sigma$-compact, locally compact groups acting on {\em connected} compact spaces $M$. 
It follows from \cite{Tits2} (or \cite{Kramer}) that $G$ must be $\PGL_2(\RR)$ (or $\PGL_2(\CC)$) acting on the real (respectively complex) projective line. 

The remaining case is that of sharply $3$-transitive actions on infinite disconnected compacta $M$. Note that the case of connected $M$ seems to be hiding some underlying structure. Indeed, we note that $\PGL_2(\RR)$ is the isometry group of the hyperbolic plane. So, $PGL_2(\RR)$ is hyperbolic and the sharply $3$-transitive action of this group coincides with its action on the hyperbolic boundary, namely $P^1(\RR)$. Similarly, $\PGL_2(\CC)$ is the orientation preserving isometry group of hyperbolic $3$-space. Again, we thus obtain a hyperbolic group and the sharply $3$-transitive action coincides with the action on the hyperbolic boundary of $\PGL_2(\CC)$, namely $P^1(\CC)$. It makes sense to ask whether any sharply $3$-transitive action of a $\sigma$-compact, locally compact group on an infinite {\em disconnected} compact space also coincides with the action of a hyperbolic group on its boundary. We show that this is indeed the case.
\begin{theorem}
Let $G$ be a $\sigma$-compact, locally compact group acting sharply $3$-transitively and continuously on an infinite disconnected compact space $M$. Then $G$ is hyperbolic and $M$ is $G$-equivariantly homeomorphic to the boundary of $G$
\end{theorem}
\begin{proof}
As $G$ acts sharply $3$-transitively on $M$, it acts $3$-cocompactly and $3$-properly on $M$ (see Lemma \ref{lemma:sharplytransisproper}). Consequently, Theorem~\ref{theorem:Bowditch} asserts that $G$ is hyperbolic and that $M$ can be identified with the hyperbolic boundary of $G$.
\end{proof}
Again looking at the case of connected $M$, it makes sense to ask the following question.
\begin{question}
Assume that $G$ is a $\sigma$-compact locally compact group acting sharply $3$-transitively on an infinite disconnected compact space $M$. Is it true that $M$ must be the projective line over a locally compact field and that $G$ is the projective group of $M$?
\end{question}
The only known examples of sharply $3$-transitive actions of $\sigma$-compact locally compact groups on infinite disconnected compact spaces $M$, and conjecturally the only examples, are the groups $\PGL_2(k)$ acting on $P^1(k)$ for $k$ a non-Archimedean local field. Our contribution to the above question is formulated in Theorem \ref{theorem:sharply3intro}. It should be read in the light that the projective line $P^1(k)$ over a non-archimedean local field $k$, is the boundary of a locally finite tree (namely the Bruhat-Tits tree associated to $PGL_2(k)$).
\begin{theoremsharply3}
Let $G$ be a $\sigma$-compact, locally compact group acting sharply $3$-transitively and continuously on an infinite disconnected compact space $M$.
Then $G$ acts continuously, properly and vertex-transitively by automorphisms on a locally finite tree $T$ such that there is a $G$-equivariant homeomorphism $f:\partial T \rightarrow M$.
\end{theoremsharply3}
\begin{proof}
We have already shown that $G$ is hyperbolic and that $M$ can be identified with the hyperbolic boundary of $G$. The action on the boundary is sharply $3$-transitive and so it is in particular transitive and faithful. It follows from Theorem $8.1$ in \cite{Caprace} that $G$ is a standard rank 1 group. Since $M$ is disconnected, it follows that $G$ acts properly and cocompactly on a locally finite tree $T$ whose boundary is equivariantly homeomorphic to $M$ (see Theorem D in \cite{Caprace}, alternatively one can use our Theorem \ref{theorem:tree_groups}).

It remains to show that $T$ can be chosen so that the action is vertex-transitive. Indeed, one can choose $T$ to have all vertices of degree $\geq 3$ by repeatedly ``erasing'' all vertices of degree 2 and removing all vertices of degree 1. In that case, the action will be transitive on the vertices of $T$, since it is $3$-transitive on the boundary and each point in the tree is the unique centre of some triple of distinct points of $M$.
\end{proof}

\begin{remark} Once it is observed that the action of $G$ on $M$ is $3$-proper, it is not hard to reduce Theorem~\ref{theorem:sharply3intro} to Theorem~\ref{theorem:transitiveconvergencegroup}. However, the proof of the latter makes essential use of both parts of Theorem~\ref{theorem:BowBowditch+} while in the former, $3$-cocompactness follows immediately from $3$-transitivity.
\end{remark}


\section{Transitive convergence groups} \label{sc:transitiveconvergence}

The goal of this section is to prove Theorem \ref{theorem:transitiveconvergencegroup}. Let us start by collecting some observations about convergence groups that are well known, at least for discrete groups \cite[Section 2]{Bo4}. We warn the reader that some elementary properties of discrete convergence groups do not carry over to the non-discrete case. For example, for a certain (rather restricted) class of non-discrete groups, a conical limit point can also be a bounded parabolic point (this happens for example with $\Isom(\HH^2) \actson \partial \HH^2 = S^1$).

\begin{notation} Throughout this section, $G$ is a locally compact group acting continuously and $3$-properly on an infinite compact set $M$.
\end{notation} 
As in section $2$ of \cite{Bo4}, we say that an element $\gamma \in G$ is \textbf{elliptic} if the cyclic subgroup generated by $\gamma$ has compact closure. If an element $\gamma$ is not elliptic, then it fixes one or two points in $M$ and we call $\gamma$ \textbf{parabolic} or \textbf{loxodromic} respectively. Moreover, for any loxodromic (resp. parabolic) element $\gamma$ we have that $(\gamma^i)|_{M \backslash \{x\}} \to y$ as $i \to \infty$ where $x$ and $y$ are the two fixed points of $\gamma$ (resp. $x = y$ is the unique fixed point of $\gamma$). We call $x$ and $y$ the \textbf{repelling} and \textbf{attracting} fixed points of $\gamma$ respectively. Note that, in the loxodromic case, both $x$ and $y$ are conical limit points.

We will make use of the following criterion for an element $\gamma \in G$ to be loxodromic. 
\begin{lemma} Suppose that $U\subset M$ is open and suppose that $\gamma \overline U$ is a proper subset of $U$. Then $\gamma$ is loxodromic.
\end{lemma}
\begin{proof} First we show that $\gamma$ is not elliptic. 
Let $x \in U \backslash \gamma \overline U$. For any integer $i > 0$ we have $\gamma^i(x) \notin U \backslash \gamma \overline U$ which is an open neighborhood of $x$. Interchanging the roles of $U$ and $M \backslash \gamma \overline U$ one sees that the same holds for $i <0$. It follows that the identity is not in the closure of $\{\gamma^i \mid i\neq 0\}$. In particular the identity is isolated in $\langle \gamma \rangle$ and $\gamma$ has infinite order. Thus $\langle \gamma \rangle$ is an infinite discrete subgroup of $G$, so that $\gamma$ is not elliptic. 

It remains to show that $\gamma$ has two fixed points. Observe that $U^+=\cap_{i > 0} \gamma^i \overline{U}$ and $U^- = \cap_{i < 0} \gamma^i (M \backslash U)$ are nonempty, disjoint, $\gamma$-invariant compact subsets of $M$. One then applies the convergence property to show that $U^+$ and $U^-$ are each reduced to a single point.
\end{proof}
As an immediate consequence we get the following
\begin{lemma} \label{lemma:loxodromiccollapsing}
	Suppose $(\gamma_i)_i$ is a net and $x \neq y$ such that $\gamma_i|_{M \backslash \{x\}} \to y$ then $\gamma_i$ is loxodromic for all sufficiently large $i$.
\end{lemma}
\begin{proof} Let $V, W$ be disjoint compact neighborhoods of $x$ and $y$ respectively such that $M \backslash (V \cup W) \neq \emptyset$. Letting $U = M \backslash V$ then for sufficiently large $i$ we have $\gamma_i(\overline{U}) \subset W$, which is a proper subset of $U$.
\end{proof}

The crucial observation needed to prove Theorem~\ref{theorem:transitiveconvergencegroup} is to show the existence of a conical limit point.
\begin{lemma} \label{lemma:existenceCLP}
	Suppose $G$ is non-compact and acts on $M$ without global fixed point. Then there exists a loxodromic element $\gamma \in G$, and in particular there exists a conical limit point $x \in M$.
\end{lemma}
\begin{proof} 
	Since $G$ is not compact, there is a sequence $(\gamma_i)_i$ leaving every compact set. After passing to a collapsing subnet $(\gamma_i)_i$ there are two points $x,y \in M$ such that $\gamma_i|_{M \backslash \{x\}} \to y$.
	\begin{itemize}
		\item Suppose $x \neq y$ then by Lemma \ref{lemma:loxodromiccollapsing} any $\gamma_i$ is loxodromic for $i$ sufficiently large.
		\item Suppose $x = y$. Take $\gamma \in G$ such that $\gamma(y) \neq y$. We observe that $(\gamma \gamma_i)|_{M \backslash \{x\}} \to \gamma(y) \neq x$. Arguing as in the previous case, there exists a loxodromic element.\qedhere
	\end{itemize}
\end{proof}

We shall also need the following 
\begin{prop} Suppose $G$ is non-compact and acts transitively on $M$. Then any compact normal subgroup of $G$ acts trivially on $M$.
\end{prop}
\begin{proof}
	Let $K \lhd G$ be a compact normal subgroup. By the previous lemma, there exists a loxodromic element $\gamma \in G$ with $x, y$ as repelling and attracting fixed points respectively. 
	\begin{claim*} $K$ fixes $y$
	\end{claim*} Seeking a contradiction, we suppose there is some $\kappa \in K$ such that $\kappa(y) \neq y$. Take two distinct points $u_1 \neq u_2$ in $M \backslash \{x\}$. Let $U, V, W$ be neighborhoods of $x,y$ and $\kappa(y)$ respectively. Up to making $V$ and $W$ smaller, we can assume that $V$ and $W$ are disjoint, and that $\kappa(V) \subset W$. Now for all $i$ sufficiently large we have $\gamma^i (u_j) \in V$ and $\gamma^{-i} (W) \subset U$ for $j=1,2$. Since the neighborhood $U$ of $x$ is arbitrary we have shown that $\gamma^{-i} \kappa \gamma^i (u_j) \to x$ as $i \to \infty$ for $j = 1,2$. This means that $\gamma^{-i} \kappa \gamma^i (u_1,u_2)$ leaves every compact subset of $M^{(2)}$ which contradicts the fact that $K$ is a compact normal subgroup.
	
	Now it is easy using transitivity of $G$ to show that any $y \in M$ is the attracting fixed point of some loxodromic element $\gamma \in G$. Thus the claim implies that $K$ acts trivially on $M$.
\end{proof}
\begin{corollary} \label{corollary:Mmetrizable} If $G$ is $\sigma$-compact, non-compact and acts faithfully and transitively on $M$, then $G$ and $M$ are metrizable.
\end{corollary}
\begin{proof} It is a theorem of Kakutani and Kodaira~\cite{KK} that any $\sigma$-compact locally compact group $G$ has a compact normal subgroup $K \lhd G$ such that $G / K$ is second countable. Since $G$ acts faithfully, any compact normal subgroup must be trivial by the preceding lemma. Thus $G$ is second countable. Fixing $x \in M$ we see by Lemma~\ref{lemma:Baireargument} that the space $M \cong G / G_x$ is a continuous open image of $G$ so that the compact space $M$ is second countable, and hence metrizable.
\end{proof}

We now have all the tools needed to reduce Theorem~\ref{theorem:transitiveconvergencegroup} to Theorem~\ref{theorem:BowBowditch+}.
\begin{proof}[Proof of Theorem~\ref{theorem:transitiveconvergencegroup}]
	Let $G$ be a non-compact, $\sigma$-compact locally compact group acting continuously, faithfully, transitively and $3$-properly on a compact set $M$. By Lemma~\ref{lemma:existenceCLP} there exists a conical limit point $x \in M$, hence all points are conical limit points by transitivity of $M$. Moreover $M$ is metrizable by Corollary~\ref{corollary:Mmetrizable}. Thus Theorem~\ref{Theorem:Bowditch+} implies that $G$ acts $3$-cocompactly on $M$, so that $G$ is a hyperbolic group acting transitively and faithfully on its boundary by Theorem~\ref{theorem:Bowditch}. One can thus appeal to the characterization of boundary transitive hyperbolic groups below to finish the proof.
\end{proof} 

The last step of the above proof is the characterization of boundary transitive hyperbolic groups due to Caprace, Cornulier, Monod and Tessera.
\begin{theorem}[{\cite[Theorem 8.1]{Caprace}}] \label{theorem:boundarytransitivehyperbolic}
	Let $G$ be a non-elementary hyperbolic locally compact group. Then the following are equivalent:
	\begin{enumerate}
		\item $G$ acts faithfully and transitively on its boundary.
		\item $G$ acts faithfully and $2$-transitively on its boundary.
		\item $G$ is a standard rank one group.
	\end{enumerate}
\end{theorem}


\section{Uniform convergence groups on the Cantor Set} \label{sec:tree_groups}
	
	The goal of this section is to prove Theorem~\ref{theorem:tree_groups}, which characterizes locally compact uniform convergence groups on perfect totally disconnected compact sets as those locally compact groups acting continuously, properly and cocompactly on locally finite trees. Remark that \eqref{cond:ucg_cantor} $\Rightarrow$ \eqref{cond:hyptdbdry} is a particular case of Theorem \ref{theorem:Bowditch}, and conversely any non-elementary hyperbolic group acts as a uniform convergence group on its (perfect) boundary (see Section \ref{sc:Bowditch}). Furthermore,  it is straightforward that \eqref{cond:treegp} $\Rightarrow$ \eqref{cond:qitotree} $\Rightarrow$ \eqref{cond:hyptdbdry}. We give some clarifications needed for the remaining implication \eqref{cond:hyptdbdry}  $\Rightarrow$  \eqref{cond:treegp}, which we restate below.
	
	The tools needed are a combination of Dicks and Dunwoody's theory of structure trees \cite{DicksDun} and a graph-theoretical definition of \emph{accessibility} due to Thomassen and Woess \cite{ThomaWoess}. Some of the ideas below were first explicited for locally compact groups by Kr\"on and M\"oller \cite{KroenMoeller}. We refer the reader to the paper of Cornulier \cite{Cornulier} for a comprehensive reference on accessibility and locally compact hyperbolic groups with a totally disconnected boundary.
	\begin{prop}
		Suppose $G$ is a non-elementary hyperbolic locally compact group with totally disconnected boundary $\partial G$. Then $G$ acts continuously, properly and cocompactly on a locally finite tree $T$.
	\end{prop}
	\begin{proof}
	We proceed in three steps.
	\begin{step} The group $G$ acts continuously, properly and cocompactly on a locally finite graph $X$.
	\end{step} 
	Indeed, the connected component of the identity $G_0$ is acts trivially on the totally disconnected boundary of $G$, so that $G_0$ is compact. Thus one may take $X$ to be a graph as in Proposition \ref{prop:nice_space}.
	\begin{step} $G$ acts continuously and cocompactly on a tree $T$ such that edge stabilizers are compact and vertex stabilizers are compactly generated with at most one end.
	\end{step}
	Since $X$ is hyperbolic, the ends of $X$ are in natural bijection with the connected components of $\partial X$, which are exactly the points of $\partial X$. It follows that $X$ is a vertex-transitive graph with no thick ends (no geodesic has both endpoints in the same end), so that $X$ is an accessible graph in the sense of Thomassen and Woess \cite{ThomaWoess}. In particular they show how to construct the desired tree $T$ using the structure trees of Dicks and Dunwoody \cite{DicksDun}.
	\begin{step} $G$ acts on $T$ with compact vertex stabilizers. \label{step:compact_vertex_stabs}
	\end{step}
	Since $G$ acts on the tree $T$  with compact edge stabilizers, for each vertex $v$ of $T$ the vertex stabilizer $G_v$ is quasi-isometrically embedded in $G$. In particular each such $G_v$ is hyperbolic, and the inclusion $G_v \hookrightarrow G$ induces an embedding $\partial G_v \hookrightarrow \partial G$. In conclusion, $G_v$ is a hyperbolic group with at most one end but with a totally disconnected boundary. However a one-ended hyperbolic group has a connected boundary with at least 2 points, so that $G_v$ is in fact $0$-ended, or in other words compact.
	
	It now follows from Step~\ref{step:compact_vertex_stabs} that $G$ acts properly on the tree $T$. We can assume, up to taking the barycentric subdivision of $T$, that $G$ acts on $T$ without inversions, so that if a vertex $v$ is an enpoint of an edge $e$, then $G_e \subset G_v$. Since $G$ acts on $T$ with finitely many edge orbits, and since edge stabilizers have finite index in vertex stabilizers containing them (being open subgroups of compact groups), it follows that $T$ is locally finite.
	\end{proof}
	
	\subsection{End-transitive graphs and groups} \label{sec:end-transitivegraphs} In this section, we sketch an alternate proof of the following theorem of Nevo
	\begin{theorem}[\cite{Nevo}] \label{theorem:Nevo} Let $X$ be a locally finite connected graph with infinitely many ends such that $\Aut(X)$ is non-compact and acts transitively on the set of ends of $X$. Then there exists a locally finite tree $T$ such that $\Aut(X)$ acts continuously, properly and cocompactly on $T$, and whose space of ends is equivariantly homeomorphic to the ends of $X$.
	\end{theorem}
		
	The group $G = \Aut(X)$ is locally compact $\sigma$-compact and acts as a convergence group on the space of ends $M$ of $X$, which is the Cantor Set. Moreover the action is assumed to be transitive, so the same reasoning as in the proof of Theorem \ref{theorem:transitiveconvergencegroup} can be used to show that $G$ is hyperbolic such that the boundary of $G$ is equivariantly homeomorphic to $M$. Note that we cannot use the characterization of boundary-transitive hyperbolic groups in \cite{Caprace} since it relies on Theorem~\ref{theorem:Nevo}. However, since the boundary of $G$ is totally disconnected, we can conclude using Theorem~\ref{theorem:tree_groups}. Note further that both Nevo's original approach to Theorem~\ref{theorem:Nevo} and ours eventually rely on the theory of Dunwoody cuts.
	
	\begin{remark} A slightly more restrictive formulation of this result is as follows : let $G$ be a compactly generated locally compact group with infinitely many ends acting transitively on its space on ends. Then $G$ has a compact normal subgroup $K$ such that $G/K$ is a totally disconnected standard rank one group, i.e. $G$ acts continuously, properly and cocompactly on a locally finite tree $T$ such that the action on $\partial T$ is $2$-transitive.
	\end{remark}

\section{$n$-proper $n$-cocompact actions for $n\geq 4$} \label{sec:nproper}
In the remainder of this section, unless explicitly mentioned otherwise, we let $G$ be a locally compact group acting continuously on a compact, metrizable set $M$. We have seen that the $n$-properness of the action can imply strong conditions on $G$ or $M$. For example, we showed in Lemma \ref{lem:P1=P2} that every $\sigma$-compact, locally compact group acting $2$-properly and $2$-cocompactly on a metrizable compactum $M$, must itself be compact. Further, in Theorem \ref{theorem:Bowditch}, we used $3$-properness, together with $3$-cocompactness, to characterize hyperbolicity. It is natural to see what happens for the case $n\geq 4$. We prove the result below, showing that then $M$ must satisfy some {\em disconnectedness-property}.
	\begin{theorem_finite_disconnecting} Suppose that $M$ is a {\em locally connected}, compact, metrizable space admitting an $n$-proper, $n$-cocompact action for some $n \geq 4$. Then there exists a set $P \subset M$ of cardinality at most $\lfloor \frac{n-1}{2} \rfloor$ such that $M \backslash P$ is not connected. 
	\end{theorem_finite_disconnecting}
	The result emphasizes the different behaviours for $n=4$ and $n=3$ respectively. Indeed, for $n=4$, the local connectedness of $M$ implies that $M$ has a global cut point. However, the cut-point conjecture (see \cite{B2}, \cite{SW1}) shows that whenever the boundary of a discrete hyperbolic group is connected, then it has no global cut points and it is locally connected (\cite{BM}).

Our result is also related to the following question that G. Mess asked B. Bowditch. Noting that there are no infinite groups acting sharply-$n$-transitively on a set $M$ for $n\geq 4$, this question occurs very naturally.
\begin{question}
Does there exist an infinite locally compact group $G$ acting continuously, $n$-properly and $n$-cocompactly on a compact, perfect metrizable space $M$ for $n \geq 4$?
\end{question}
Our result gives a necessary condition on $M$ for such an action to exist.\\

\noindent We start our exposition with a general lemma exploring the interaction of $n$-properness and $m$-cocompactness.
	\begin{lemma} \label{lemma:NpropMcocomp} Let $m > n \geq 1$ be two integers and let $G$ be a locally compact group acting continuously, $n$-properly and $m$-cocompactly on a locally compact space $M$. Then $M$ is discrete.
	\end{lemma}
	\begin{proof} Suppose that $M$ is non-discrete. Hence $M$ is infinite and contains an accumulation point $x_1$. Pick $x_2, \ldots, x_n$ in $M$ so that $x_i \neq x_j$ for $1 \leq i \neq j \leq n$. Pick also a sequence $(y_k)_k \to x_1$ such that $y_k \neq x_i$ for any $k$ and any $1 \leq i \leq n$. In particular,  $\theta_k = (y_k,x_1,\ldots,x_n) \in M^{(n+1)}$ leaves every compact subset of $M^{(n+1)}$. Since the action is $m$-cocompact, it is in particular $(n+1)$-cocompact so there is a compact subset $K \subset M^{(n+1)}$ and elements $\gamma_k \in G$ such that $\gamma_k(\theta_k) \subset K$. After taking a subnet, we can suppose that $\gamma_k(\theta_k) \to \theta = (z_1,\ldots,z_{n+1}) \in K$. Now since the $\theta_k$ leave every compact subset of $M^{(n+1)}$, the net $(\gamma_k)_k$ is wandering. But this contradicts $n$-properness, since $\gamma_k (x_1, \ldots, x_n) \to (z_2, \ldots, z_{n+1}) \in M^{(n)}$.
	\end{proof}

In order to prove Theorem~\ref{theorem:finitedisconnecting}, we elaborate on a result by Gerasimov (Lemma Pr1 in \cite{Gerasimov}). In order to formulate it, note that each $\gamma\in G$ acts as a homeomorphism of $M$ so we can consider its graph as a closed subset of $M\times M$ defined by $\graph(\gamma)=\{(x,\gamma(x))\mid x\in M\}$. As $M\times M$, equipped with the product metric, is a metric space, we can equip $\Closed(M\times M)$, the set of closed subsets of $M\times M$, with the Hausdorff metric. The induced topology on $\Closed(M\times M)$ coincides with the {\bf Vietoris topology} as used by Gerasimov.

Recall that a group $G$ acts $3$-properly on a compact metrizable set if and only if every wandering sequence has a collapsing subsequence. We can also express 
this as follows: given a wandering sequence $(\gamma_i)_i$ in $G$, there exists a subsequence $(\gamma_j)_j\subset (\gamma_i)_i$ and $b,c\in M$ such that $\graph(\gamma_j)_j$ converges in $\Closed(M\times M)$ to a subset of the set $\{b\}\sharp \{c\}:= (\{b\}\times M) \cup (M\times \{c\})$. In \cite{Gerasimov}, Gerasimov generalized this result from $3$-proper actions to $n$-proper actions for any $n\geq 2$. 
\begin{prop}[Lemma Pr1 in \cite{Gerasimov}] Let $n \geq 2$ and let $G$ be a locally compact group acting continuously by homeomorphisms on a compact metrizable space $M$. Equip $\Closed(M\times M)$ with the Vietoris topology. Then $G$ acts $n$-properly on $M$ if and only if for any wandering sequence $(\gamma_i)_i$, there is a subsequence $(\alpha_i)_i\subset (\gamma_i)_i$ and nonempty finite sets $P, Q \subset M$ with $|P| + |Q| \leq n-1$  such that $\graph(\alpha_i)_i\to Z$ in $\Closed(M\times M)$ for some subset $Z \subset P\sharp Q:=(P\times M) \cup (M\times Q)$.
\label{prop:Gerasimov}
\end{prop}
\begin{notation}
In all that follows, we let $(\gamma_i)_i\subset G$ denote a wandering sequence for which there exist nonempty finite sets $P, Q \subset M$ as above such that $\graph(\gamma_i)_i\to Z \subset P\sharp Q:=(P\times M) \cup (M\times Q)$ in Closed$(M\times M)$.
\end{notation}
		\begin{lemma}
Choose any compact, connected subset $K\subset M\backslash P$. For $x\in K$, take a subsequence $(\alpha_i)_i$ of $(\gamma_i)_i$ such that $(\alpha_i(x))_i$ converges to some $q\in Q$. Then $(\alpha_i)_i$ converges uniformly over $K$ to $q$.
\label{lemma:uniformconvergence}
		\end{lemma}
	\begin{proof}
Let $r:=\min_{a \neq b\in Q}(d(a,b))$, i.e. the distance between horizontal 
lines in $P\sharp Q$ is at least $r$. If $r=0$, then we redefine $r:=\infty$. Choose any 
$0<\epsilon <\min(r/2,d(K,P))$ and take $N\in \NN$ such that 
$d(\alpha_n(x),q)<\epsilon$ for all $n\geq N$. Enlarging $N$ if necessary, Proposition \ref{prop:Gerasimov} implies that the graph of $\alpha_n$ lies in the $\epsilon$-neighbourhood of $P\sharp Q$ for all $n\geq N$. Because $\graph((\alpha_n)_{\mid K})$ is connected, because it contains $(x,\alpha_n(x))$, and because it lies in an $\epsilon$-neighbourhood of $P\sharp Q$ but more 
than $\epsilon$ away from $P\times M$, we conclude that $\alpha_n(K)$ lies in the $\epsilon$-neighbourhood of $q$.
	\end{proof}

		\begin{prop}
If $M$ is a {\em locally connected}, compact, metrizable space, then there is a subsequence $(\alpha_i)_i$ of $(\gamma_i)_i$ which converges pointwise everywhere on $M$. Moreover, it converges pointwise to a constant function on connected components of $M\backslash P$ and it converges uniformly on compact connected subsets of $M\backslash P$.
	\label{prop:uniform}
		\end{prop}
	\begin{proof}
Choose any $x\in M\backslash P$ and take a subsequence $(\alpha_i(x))_i$ of $(\gamma_i(x))_i$ which converges to some $q\in Q$. Given $y$ in the connected component of $x$ in $M\backslash P$, we show first that it lies in a compact connected set $K\subset M\backslash P$ containing $x$. Using Lemma \ref{lemma:uniformconvergence}, this implies that $(\alpha_i(y))_i$ converges to $q$.

The connected component of $x$ in $M$ is locally path connected because it is a connected, locally connected, compact metrizable space. The connected component $C$ of $x$ in $M\backslash P$ is therefore also locally path connected. Consequently, any point $y\in C$ can be connected to $x$ via a path in $C$. This path is a compact connected subset of $M\backslash P$ containing both $x$ and $y$, as desired.

Finally, let us prove pointwise convergence on $M$. To this end, note that connected components of $M\backslash P$ are open because of local connectedness. Because $M$ is second countable, and because connected components of $M\backslash P$ are disjoint, there can be only countably many, say $C_1,C_2, \ldots$. Choose first $x_1\in C_1$ and take a subsequence $(\alpha_i^1)_i$ of $(\gamma_i)_i$ such that $(\alpha_i^1(x_1))_i$ converges to some $q_1\in Q$. Next, continue this process. So, choose $x_2\in C_2$ and take a subsequence $(\alpha_i^2)_i$ of $(\alpha_i^1)_i$ such that $(\alpha_i^2(x_2))$ converges to some $q_2\in Q$. Then choose $x_3\in C_3$ and so on. The sequences $((\alpha_i^i)(x_j))_i$  all converge to an element of $Q$. As shown earlier in the proof, the sequence $(\alpha_i^i)$ thus converges to the constant map on every connected component of $M\backslash P$. Finally, taking a last subsequence, we obtain additionally that it also converges on $P$ and so on the whole of $M$. 
	\end{proof}
	
We now have the necessary tools to give the proof of Theorem \ref{theorem:finitedisconnecting}.
	\begin{proof}[Proof of Theorem \ref{theorem:finitedisconnecting}]
Since the action is $n$-proper and $n$-cocompact for $n\geq 4$, we can assume that it is not $3$-proper by Lemma \ref{lemma:NpropMcocomp}. So, we can fix a wandering sequence $(\gamma_i)_i$ such that no subsequence is collapsing. As in Proposition \ref{prop:uniform}, take a subsequence, denoted $(\alpha_i)_i$, and sets $P,Q\subset M$ such that $(\alpha_i)_i$ converges pointwise everywhere on $M$ and $\graph(\alpha_i)\to Z \subset P\sharp Q$. We can assume without loss of generality that there is no proper subset $P' \subset P$ such that $Z \subset P'\sharp Q$, and similarly there is no proper subset $Q' \subset Q$ such that $Z \subset P \sharp Q'$. Denoting $\lvert P\rvert =r, \lvert Q \rvert=s$, we obtain $3 \leq r+s \leq n-1$. Up to exchanging $(\alpha_i)$ with $(\alpha_i^{-1})$ we can assume $s \geq r$ (in particular $s \geq 2$, and $r \leq \lfloor \frac{n-1}{2} \rfloor$). Now, if $C$ is a connected component of $M \backslash P$, then there is a point $q \in Q$ such that $\alpha_i$ converges pointwise on $C$ to $q$ (Proposition \ref{prop:uniform}). Moreover, by minimality of $Q$ each element $q \in Q$ arises as a limit in this fashion. Since $s \geq 2$ there are at least two connected components of $M \backslash P$. 
	\end{proof}
We will state one corollary of this result. Let us introduce the following definition.
\begin{defn}
Given $m \in \N_0$, we say that a topological space $M$ is {\bf $m$-homogeneous}, if the homeomorphism group of $M$ acts $m$-transitively on $M$.
\end{defn}
		\begin{corollary} \label{corollary:nproperCantor} Let $n \geq 4$, and let $m = \lfloor \frac{n-1}{2} \rfloor+2$. If $G$ acts $n$-properly, $n$-cocompactly on a {\em locally connected} metrizable $m$-homogeneous compact set $M$, then $M$ is finite. 
	\end{corollary}
	\begin{proof} 
Let $P$ be a set as in Theorem \ref{theorem:finitedisconnecting}, so $\lvert P \rvert \leq \lfloor \frac{n-1}{2} \rfloor$ and $M\backslash P$ is not connected. Assume first by contradiction that $M \backslash P$ is not totally disconnected.
	 Then, since $M \backslash P$ is not connected, there exist $C_1$ and $C_2$ two distinct connected components of $M \backslash P$ and elements $x \neq y \in C_1$ and $z \in C_2$. Since $M$ is $m$-homogeneous, there exists a homeomorphism $\varphi$ of $M$ such that $\varphi(P) = P$, $\varphi(x) = x$ and $\varphi(y) = z$. Clearly, $\varphi$ maps $x$ and $y$ to distinct connected components of $M \backslash P$, which is not possible since $\varphi$ is a homeomorphism fixing $P$: a contradiction.
	
	As $P$ is finite and $M$ is locally connected, we so conclude that $M$ is totally disconnected and so $M$ must be discrete, hence finite by compactness.
	\end{proof}
	\bibliographystyle{amsalpha}
\providecommand{\bysame}{\leavevmode\hbox to3em{\hrulefill}\thinspace}
\providecommand{\MR}{\relax\ifhmode\unskip\space\fi MR }
\providecommand{\MRhref}[2]{%
  \href{http://www.ams.org/mathscinet-getitem?mr=#1}{#2}
}
\providecommand{\href}[2]{#2}

\end{document}